%% file: main_v2.tex
\documentclass{amsart}
\usepackage{graphicx} 
\usepackage[all]{xy}
\usepackage{color}
\usepackage{enumerate}
\usepackage{float}
\usepackage{amssymb}
\usepackage{amsmath}
\usepackage{tikz}
\usepackage{tikz-cd}
\usetikzlibrary{decorations.markings}
\usepackage{comment}
\usepackage{enumitem}
\usepackage{mathabx}
\usepackage{url}
\usepackage{hyperref}

\usepackage[top=20mm,left=20mm,right=20mm]{geometry}
\newcommand{\ZZ}{\mathbb{Z}}
\newcommand{\QQ}{\mathbb{Q}}
\newcommand{\ff}{\mathcal{F}}

\newcommand{\eqMark}{$E_{0}$}
\newcommand{\eqMarktwo}{$E_{2}$}
\newcommand{\eqMarkk}{$E_{k}$}
\newcommand{\eqMarkmirror}{$E_{q+q^{-1}}$}
\newcommand{\trMark}{$\mathbb{T}_{0}$}

\newcommand{\trMarkmirror}{$\mathbb{T}_{q+q^{-1}}$}

\theoremstyle{plain}
\newtheorem{theorem}{Theorem}[section]
\newtheorem{proposition}[theorem]{Proposition}
\newtheorem{corollary}[theorem]{Corollary}
\newtheorem{lemma}[theorem]{Lemma}
\newtheorem{conjecture}[theorem]{Conjecture}
\newtheorem{remark}[theorem]{Remark}
\newtheorem{notation}[theorem]{Notation}

\newtheorem{definition}[theorem]{Definition}
\newtheorem{example}[theorem]{Example}

\title{A Mirror deformation of Markov Numbers}
\usepackage{xcolor}
\definecolor{pea}{RGB}{68, 168, 50}
\definecolor{plum}{RGB}{125, 50, 168}
\definecolor{cyan}{RGB}{71, 191, 255}
\definecolor{bordeaux}{RGB}{150, 5, 51}
\definecolor{coral}{RGB}{245, 93, 159}
\definecolor{teal}{RGB}{28, 157, 186}
\definecolor{violet}{RGB}{138,43,226}

\newcommand{\defin}[1]{{\color{coral}\emph{#1}}}

\author[L. Bittmann]{Léa Bittmann}
\address{Institut de Recherche Mathématique Avancée, UMR 7501 Université de Strasbourg et CNRS, 7 rue René-Descartes, 67000 Strasbourg, France}
\email{lea.bittmann@math.unistra.fr }

\author[P. Jouteur]{Perrine Jouteur}
\address{Université de Reims Champagne-Ardenne, Laboratoire de Mathématiques, CNRS, UMR 9008, Reims, France}
\email{perrine.jouteur@univ-reims.fr}

\author[E. Kantarcı Oğuz]{Ezgi Kantarcı Oğuz}
\address{Galatasaray Üniversitesi, Ortaköy, Çırağan Cd. No:36, 34349 Beşiktaş, İstanbul, Türkiye}
\email{ezgikantarcioguz@gmail.com}

\author[M. Molander]{Melody Molander}
\address{Department of Mathematics, The Ohio State University, 100 Math Tower, 231 W 18th Ave, Columbus, OH 43210, USA}
\email{molander.3@osu.edu}

\author[E. Yıldırım]{Emine Yıldırım}
\address{International Center for Mathematical Sciences - Sofia, Bulgarian Academy of Sciences, Acad. G. Bonchev Str., Bl. 8, Sofia
1113, Bulgaria}
\email{e.yildirim@math.bas.bg}


\begin{document}

\begin{abstract} We introduce a \emph{deformed squared Markov equation} given by $X^2 + Y^2 + Z^2 + (q+q^{-1})(XY+YZ+XZ) = 3(1 + q + q^{-1})XYZ$. Symmetric solutions of this new equation present a remarkable factorization property which allows us to talk about their square roots. These square roots give a natural $q$-deformation of the Markov numbers that has not previously occurred in the literature. We call them \textit{mirror Markov numbers}. We prove a characterization of mirror Markov numbers and discover a mutation rule, \textit{mirror mutation}, to generate them all. We also prove a geometric realization of the corresponding mirror mutation on a once-punctured sphere with three orbifold points. Our mirror deformation leads to deformations of Fibonacci and Pell branches for which we give precise formulas. Furthermore, the deformed squared Markov equation specializes to many other very well known generalized Markov equations. We also obtain the super Markov numbers from a specialization of the deformed squared Markov numbers, which we use to prove a conjecture of Musiker.
\end{abstract}

\maketitle

\section{Introduction}

The Markov Diophantine equation, defined by Markov in the late 19th century~\cite{M79} is given by 
\begin{equation}\label{eq:originalmarkov}
x^2+y^2+z^2=3xyz.\tag{\eqMark}
\end{equation}

The triples of positive integer solutions to this equation are called \defin{Markov triples}, and each number that occurs in such a triple is called a \defin{Markov number}. Given a Markov triple, one can obtain a new triple via replacing an entry $x$ on a triple with $x'=(y^2+z^2)/x$, an operation called a \defin{Vieta jump}. The new triple $(x',y,z)$ is also a Markov triple. Furthermore, this operation can be applied to any coordinate as the equation is symmetric on $x$, $y$ and $z$. 

All solutions of the Markov equation can be obtained from the initial solution $(1,1,1)$ via Vieta jumping. We will, as is often done in the literature, visualise the Markov triples as vertices on a graph where the edges represent the Vieta jumps. This is called the \defin{Markov tree} \trMark. It is non-trivial (though not particularly complicated) to show that the solutions form a connected structure where starting with the solution $(1,2,5)$ we get an infinite binary tree. We will prove a similar construction for a deformed version of the equation in this work from which the reader may infer the proof for the classical case. See Aigner's book \cite{Aigner2013MarkovBook} for more on the classical Markov tree.

One property we would like to underline about the Markov tree before we move on to the deformed case concerns the maximum numbers that appears in each triple. One can see in Figure~\ref{fig:classical markov tree} that maxima increase as we move further from the root of the tree. This property, which occurs in the rest of the tree as well, is significant as the maxima of the Markov triples have an important role in the literature. They are in bijection with discrete parts of Markov and Lagrange spectrums, concerned with minima of binary quadratic forms and rational approximations respectively. Frobenius conjectured in 1913 that each of these maxima is unique, i.e. they do not occur as a maximum in another Markov triple. This conjecture is still open after a century (see Aigner's book \cite{Aigner2013MarkovBook}). 

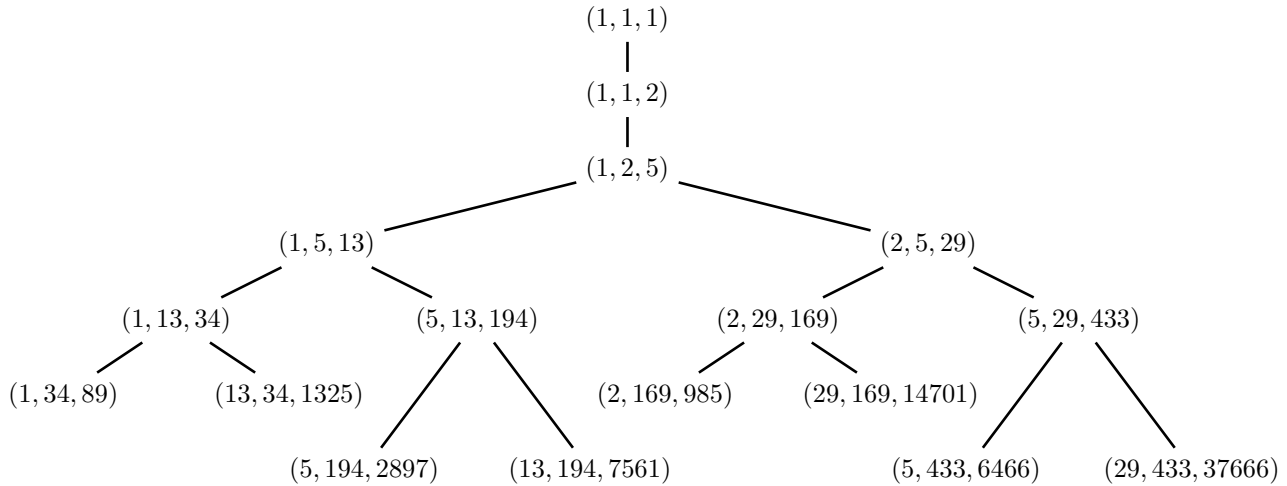
\begin{figure}[H]
    \centering
    \begin{tikzpicture}
            \node (A0) at(0,2) {$(1,1,1)$};
            \node (A1) at(0,1) {$(1,1,2)$};

            \node (A) at(0,0) {$(1,2,5)$};
            
            \node (B1) at(-4,-1) {$(1,5,13)$};
            \node (B2) at(4,-1) {$(2,5,29)$};
            
            \node (C1) at(-6,-2) {$(1,13,34)$};
            \node (C2) at(-2,-2) {$(5,13,194)$};
            \node (C3) at(2,-2) {$(2,29,169)$};
            \node (C4) at(6,-2) {$(5,29,433)$};
            
            \node (D1) at(-7.5,-3) {$(1,34,89)$};
            \node (D2) at(-4.5,-3) {$(13,34,1325)$};
            \node (D3) at(-3.5,-4) {$(5,194,2897)$};
            \node (D4) at(-0.5,-4) {$(13,194,7561)$};
            \node (D5) at(0.5,-3) {$(2,169,985)$};
            \node (D6) at(3.5,-3) {$(29,169,14701)$};
            \node (D7) at(4.5,-4) {$(5,433,6466)$};
            \node (D8) at(7.5,-4) {$(29,433,37666)$};

            \draw[line width=1pt] (A0)--(A1);
            \draw[line width=1pt] (A1)--(A);

            \draw[line width=1pt] (A)-- (B1);
            \draw[line width=1pt] (A)-- (B2);
            
            \draw[line width=1pt] (B1)-- (C1);
            \draw[line width=1pt] (B1)-- (C2);
            \draw[line width=1pt] (B2)-- (C3);
            \draw[line width=1pt] (B2)-- (C4);
            
            \draw[line width=1pt] (C1)-- (D1);
            \draw[line width=1pt] (C1)-- (D2);
            \draw[line width=1pt] (C2)-- (D3);
            \draw[line width=1pt] (C2)-- (D4);
            \draw[line width=1pt] (C3)-- (D5);
            \draw[line width=1pt] (C3)-- (D6);
            \draw[line width=1pt] (C4)-- (D7);
            \draw[line width=1pt] (C4)-- (D8);
            
\end{tikzpicture}
\caption{First levels of Markov tree \trMark}
\label{fig:classical markov tree}
\end{figure}

Recent work has showcased previously unknown connections to cluster algebras where the Vieta jumps correspond to cluster mutations. This sparked new interest in both the equation and its generalizations (see for instance \cite{G22,gyoda_uniqueness_2024,kantarci_oguz_oriented_2025,evans_arithmetic_2025,Banaian_Gyoda_2025}).
 The classical Markov equation can be translated into the language of cluster algebras by considering the quiver $Q$ depicted below as Figure~\ref{fig:quiver}.

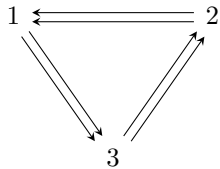
\begin{figure}[h]
\centering
\begin{tikzcd}[arrow style=tikz,>=stealth,row sep=4em]
1 
\arrow[dr,shift left=.4ex]
\arrow[dr,shift right=.4ex,""]
&& 2
\arrow[ll,shift left=.4ex,""]
\arrow[ll,shift right=.4ex,""]
\\
& 3 \arrow[ur,shift left=.4ex]
\arrow[ur,shift right=.4ex,""]
\end{tikzcd}
\caption{Markov quiver}
\label{fig:quiver}
\end{figure}

We assign the indeterminates $X_1,X_2$ and $X_3$ to the vertices of $Q$ as the initial variables of our cluster algebra and generate new variables via mutations. On the quiver level, the action of the mutation is just a reversal of the arrows for this case, so that we always get the same quiver up to a relabelling of the vertices. 
On the level of the variables, we get the following mutation formula:

\[X'_k = \frac{\displaystyle\prod_{i\to k}X_i + \displaystyle\prod_{k\to j}X_j}{X_k} =\frac{X_i^2 +X_j^2}{X_k}\]

\noindent where products are over incoming/outgoing arrows at vertex $k$.

Note that mutating a cluster $(x,y,z)$ at a coordinate, say $x$, results in a new cluster $\left(\frac{y^2+z^2}{z},y,z\right)$, corresponding exactly to a Vieta jump in the Markov equation. This gives us a correspondence between Markov triples and clusters for the cluster algebra of the quiver from Figure~\ref{fig:quiver}. In particular, every Markov number can be realized as a cluster variable evaluated at $X_1=X_2=X_3=1$. This connection can be exploited by carrying geometric or other structure from clusters to the Markov level to obtain more generalized versions of the Markov equation. Another possibility is to consider Diophantine equations coming from other cluster algebras or their generalizations.

 In our work we will do a marriage of the two methods, and look at a $q$-deformation that incorporates geometrical information into a generalized version of the Markov equation that exhibits algebraically and combinatorially interesting properties, while carrying a strong connection to the original Markov equation. 

Let us consider a new equation, the \defin{deformed squared Markov equation}, as below,

\begin{equation}
 X^2 + Y^2 + Z^2 + (q+q^{-1})(XY+YZ+XZ) = 3(1 + q + q^{-1})XYZ \tag{\eqMarkmirror}
\end{equation}
which was first suggested to us by Frédéric Chapoton.

A solution to (\eqMarkmirror) is given by a triple of Laurent polynomials $(X(q),Y(q),Z(q))$, where we require our polynomials to be \defin{symmetric} on $q$ and $q^{-1}$. We will show that all such solutions can be obtained from the initial solution $(1,1,1)$ via mutations, giving us a corresponding tree structure \trMarkmirror, see Figure~\ref{fig:deformed squared markov tree}.

The deformed squared Markov equation (\eqMarkmirror) has interesting connections to other generalizations of the Markov equation as well.

\begin{enumerate}
\item[(i)] Specializing $q+q^{-1}=\varepsilon$ in the equation \eqMarkmirror, we get the super Markov equation appearing in~\cite{M25,Huang2023}. When we let $\varepsilon$ be the product of odd variables in the context of super Teichmüller theory, we have $\varepsilon^2=0$;
\item[(ii)] Specializing $q+q^{-1}$ to integers, we obtain Gyoda-Matsushita~\cite{GM23} Markov equations;
\item[(iii)] Specializing $q$ to complex numbers, we see a connection to complex Markov numbers in the sense of ~\cite{GKW24}.
\item[(iv)] Setting $q+q^{-1}=\lambda_p=2\cos{(\pi/p)}$, we get the generalized Markov equation coming from cluster algebras of orbifold surfaces.
\end{enumerate}

Additionally, we show that solutions of the deformed squared Markov equation (\eqMarkmirror) are of the form \[(X(q),Y(q), Z(q))=(x(q)x(q^{-1}),y(q)y(q^{-1}), z(q)z(q^{-1}))\] where $x(q)$, $y(q)$ and $z(q)$ are polynomials in $q$ and $(x(1),y(1),z(1))$ is a Markov triple.

The product $X(q)=x(q)x(q^{-1})$ can be rewritten as $q^{-\operatorname{deg}(x)}x(q)\tilde{x}(q)$ where $\tilde{x}$ is the \defin{mirror image} of $x$, i.e., $x$ with its coefficients reversed. 

The mutation operation can also be defined directly on the polynomial triples $(x(q),y(q),z(q))$. This gives us a new $q$-deformation for Markov numbers that we call the \defin{mirror deformation}. Our deformation is distinct from the deformations previously explored in \cite{Leclere2021Q-DeformationsNumbers,Kogiso2020,kantarci_oguz_oriented_2025, EJMGO25}. In particular, neither the constant nor the highest degree terms are necessarily $1$, meaning they are not keeping track of a lattice structure attached to the mutations. They are a different statistic altogether. 

Regarding connection with cluster algebras from orbifold surfaces, we associate a geometric model of a sphere with one puncture and three orbifold points, appearing in \cite{FSM12,CS14,G22, BS24} to the mirror Markov combinatorics. This leads us to develop and prove an interpretation of our mutation rule on mirror deformations.

In Section~\ref{sec:solutionTree}, we prove that all solutions to the deformed squared Markov equation~(\eqMarkmirror)  are in the deformed squared Markov tree. We connect Equation~(\eqMarkmirror) to other generalized Markov equations in Section~\ref{sec:spec}. Using these connections, we state a positivity result and prove a conjecture of \cite{M25}. After giving a mutation rule for mirror deformed Markov triples, we delve into the solution tree for the mirror deformed Markov triples in Section~\ref{sec:mirrorMarkov}, with its geometric interpretation in Section~\ref{subsect:geometry}. Following the proofs from the previous section, we particularly investigate the Fibonacci and Pell branches of the mirror deformed Markov tree in Section~\ref{sec:FP}. We finish our paper by listing some open problems and elaborating on further directions in Section~\ref{sec:open}.

\subsection*{Acknowledgement} E.Y. is supported by Bulgarian National Science Fund KL-06-N92/5, Ministry of Education and Science of the Republic of Bulgaria, grant DO1-239/10.12.2024 and Simons Foundation, grant SFI-MPS-T-Institutes-00007697. E. K. O. thankfully acknowledges support by The Scientific and Technological Research Council of Türkiye  ARDEB 1001 grant 123F121.

The authors would like to thank Frédéric Chapoton for sharing the starting question of the paper, and  Banff International Research Station and the organisers of the 2025 workshop on Women in Noncommutative Algebra and Representation Theory 4 for such a wonderful event.

\section{Deformed squared Markov equation}~\label{sec:solutionTree}
Let $q$ be an indeterminate, and consider the equation 

\begin{equation}
\label{eq:mirrorMarkov}
 X^2 + Y^2 + Z^2 + (q+q^{-1})(XY+YZ+XZ) = 3(1 + q + q^{-1})XYZ. \tag{\eqMarkmirror}
\end{equation}

A \emph{solution} for this equation is, by definition, a triple of Laurent polynomials $(A(q), B(q), C(q))$ with integer coefficients satisfying \eqref{eq:mirrorMarkov}. Given a solution $(A, B, C)$ with non-zero entries, we can apply mutations to get other solutions. The mutation at $A$ replaces $A$ with

\begin{equation}
\label{eq:squared mutation rule}
\displaystyle A'=\frac{B^2+(q+q^{-1})BC + C^2}{A}=3 (1+q+q^{-1})BC-(q+q^{-1})(B+C) -A.
\end{equation}

The case $q=1$ of the deformed squared Markov equation gives a generalized Markov equation previously studied in~\cite{BS24,G22,GM23}.
\begin{equation}\label{eq:2Markov}\tag{\eqMarktwo}
x^2 + y^2 + z^2 + 2(xy+yz+xz) = 9xyz.
\end{equation}
This equation can be rewritten as 
\begin{equation*}
(x + y + z)^2 = 9xyz,
\end{equation*}
where plugging in $a^2,b^2$ and $c^2$ for $x$, $y$ and $z$ respectively gives us the original Markov equation squared:
\begin{equation*}
(a^2 +b^2 +c^2)^2 =(3abc)^2.
\end{equation*}

Now it is natural that the squares of Markov triples give solutions to the generalized equation. In fact, these are the only solutions, as shown in~\cite[Theorem 11 of Section 2]{GM23}.

\subsection{Deformed squared Markov tree}

We call the graph of all solutions of the deformed squared Markov equation~\eqref{eq:mirrorMarkov} which are obtained starting from $(1,1,1)$ by the mutation in Equation~\eqref{eq:squared mutation rule} \emph{the deformed squared Markov tree}. We justify that this is a tree because the graph reduces to the usual Markov tree when we set $q+q^{-1}=0$.

Note that as we mutate solutions to Equation~\eqref{eq:mirrorMarkov}, the exponents of $q$ will grow. We will have a higher degree polynomial in each triple; let us call it \emph{maximum} as in the classical case. Figure~\ref{fig:deformed squared markov tree} indicates only the maxima of triples in the first layers of the deformed squared Markov tree. 

\begin{remark}
    We note that the mutation in Equation~\eqref{eq:squared mutation rule} preserves symmetry in $q\leftrightarrow q^{-1}$ and it is clear from the rule that it also preserves integer coefficients. 
\end{remark}

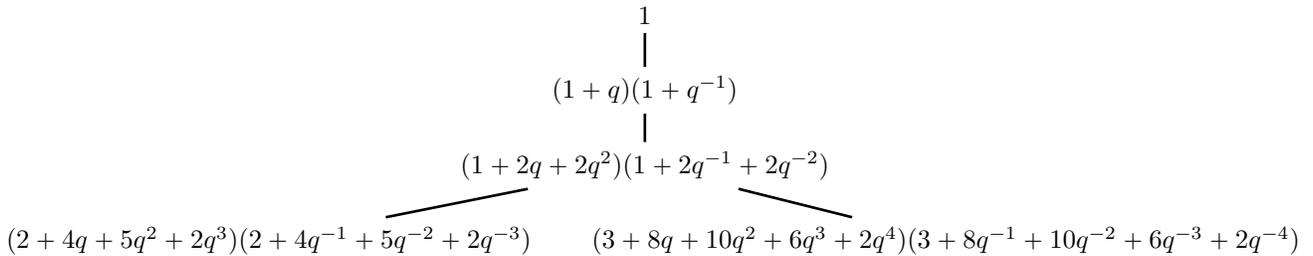
\begin{figure}[H]
    \centering
    \begin{tikzpicture}
            \node (A0) at(0,2) {$1$};
            \node (A1) at(0,1) {$(1+q)(1+q^{-1})$};

            \node (A) at(0,0) {$(1+2q+2q^2)(1+2q^{-1}+2q^{-2})$};
            
            \node (B1) at(-5,-1) {$(2+4q+5q^2+2q^3)(2+4q^{-1}+5q^{-2}+2q^{-3})$};
            \node (B2) at(4,-1) {$(3 + 8q + 10q^{2} + 6q^{3} + 2q^{4})(3 + 8q^{-1} + 10q^{-2} + 6q^{-3} + 2q^{-4})$};
           
            \draw[line width=1pt] (A0)--(A1);
            \draw[line width=1pt] (A1)--(A);

            \draw[line width=1pt] (A)-- (B1);
            \draw[line width=1pt] (A)-- (B2);
\end{tikzpicture}
\caption{Initial steps of the deformed squared Markov tree}
\label{fig:deformed squared markov tree}
\end{figure}

\begin{theorem}\label{thm:positivity}
    For any entry $(A(q),B(q),C(q))$ on the deformed squared Markov tree, $A(q)$, $B(q)$ and $C(q)$ are Laurent polynomials with positive integer coefficients that are symmetric in $q\leftrightarrow q^{-1}$.
\end{theorem}

The proof of Theorem~\ref{thm:positivity} is postponed to Section~\ref{sec:spec}.

\begin{proposition} \label{prop:allsolutionsontree}
    The symmetric solutions of the deformed squared Markov equation~\eqref{eq:mirrorMarkov} with positive integer coefficients are exactly the vertices of the deformed squared Markov tree.
\end{proposition}

Our proof for the proposition will be via induction on the maximum degree of the triple $(A(q),B(q),C(q))$. We will first start with a lemma to show that except for $(0,0,0)$, there are no solutions where one or more entries are equal to zero, so we do not need to worry about the case of a maximum degree being $-\infty$.

\begin{lemma}\label{lem:symmetriczerosolutions} Let $(A(q),B(q),C(q))$ be a solution to  Equation~\eqref{eq:mirrorMarkov} and suppose that $A, B$ and $C$ are symmetric in $q\leftrightarrow q^{-1}$. If $C(q)=0$, then $A(q)=B(q)=0$.
\end{lemma}

\begin{proof} In the case $C(q)=0$, $A(q)$ and $B(q)$ satisfy:
\begin{equation}\label{eq:Eczero}\tag{$E_{c=0}$}
A(q)^2 + B(q)^2 + (q+q^{-1})(A(q)B(q)) = 0.
\end{equation}
For a contradiction, we assume that $A(q)$ and $B(q)$ are not both equal to zero, and the degree of $A(q)$ is weakly higher than the degree of $B(q)$. 
Let $A(q)=a_tq^t+a_{t-1}q^{t-1}+\cdots a_{-(t-1)}q^{-(t-1)}+a_{-t}q^{-t}$ where $a_t\neq 0$ and $B(q)=b_tq^t+b_{t-1}q^{t-1}+\cdots b_{-(t-1)}q^{-(t-1)}+b_{-t}q^{-t}$ where the coefficients are possibly equal to $0$.

Consider the coefficients of $q^{2t+1}$ and $q^{2t}$ in \eqref{eq:Eczero}. $q^{2t+1}$ gives us: $a_tb_t=0\Rightarrow b_t=0$. With the added assumption that $b_t=0$, $q^{2t}$ gives us $a_t^2+a_t(b_{t-1})=0 \Rightarrow b_{t-1}=-a_t$.

\textbf{Claim:} For all $i\geq 1$ we have $b_{t-i}=-a_{t-i+1}$.
\textit{Proof of Claim.} We have already shown that the claim holds for $i=1$. Assume it holds for $i \leq k$ for some $k$. Now consider the coefficient of $q^{2t-k}$ on \eqref{eq:Eczero}. As no lesser terms can contribute, we can replace our $A$ and $B$ with 
\begin{eqnarray*}
    \hat{A}(q)&=&a_tq^t+a_{t-1}q^{t-1}+\cdots + a_{t-k}q^{t-k}.\\
     \hat{B}(q)&=&-a_tq^{t-1}-a_{t-1}q^{t-2} -\cdots -a_{t-k+1}q^{t-k}+ b_{t-k-1}q^{t-k-1}.
\end{eqnarray*}
The contribution of $\hat{B}^2$ to the coefficient of $q^{2t-k}$ is cancelled out by the contribution of $q^{-1}\hat{A}\hat{B}$. The contribution of $\hat{A}^2$ cancels out that of $q\hat{A}\hat{B}$ except for the terms: $a_ta_{t-k}q^{2t-k}$ from  $\hat{A}^2$ and 
: $-a_{t}b_{t-k-1}q^{2t-k}$ from $q\hat{A}\hat{B}$. So we get $a_ta_{t-k}=-a_{t}b_{t-k-1}\rightarrow a_{t-k}=-b_{t-k-1}$, proving our claim.

This can not happen since $A(q)$ and $B(q)$ are assumed to be symmetric, and we have $a_t=a_{-t}\neq 0= b_{-t-1}$.
\end{proof}

\begin{remark}
    Note that the criteria for solutions of Equation~\eqref{eq:mirrorMarkov} to be symmetric is necessary in Lemma \ref{lem:symmetriczerosolutions}. Suppose $(A(q),B(q),C(q))$ is a solution of \eqref{eq:mirrorMarkov} with integer coefficients, $A(q)\neq 0$, and $C(q)=0$. Then dividing \eqref{eq:Eczero} by $A(q)^2$ and solving the quadratic formula results in
    \begin{align*}
        \frac{B(q)}{A(q)}=-q^{-1}, \text{ or }-q.
    \end{align*}
    As an example, $(1,-q,0)$ is a nonsymmetric solution. 
There are also nonsymmetric solutions without any zeroes, such as $(1,-q,-2q^2-4q-q^{-1}-2)$. In these cases, it is possible to get negative coefficients. One can think of the symmetry condition as analogous to the positivity condition for the non $q$-deformed equations.
\end{remark}

\begin{proof}[Proof of Proposition~\ref{prop:allsolutionsontree}] 
Let $(A(q),B(q),C(q))$ be a triple of symmetric Laurent polynomials with integer coefficients solving \eqref{eq:mirrorMarkov} with (maximum) degrees $a,b,c$ respectively.

A triple where all degrees are equal is called \defin{degree-singular}.

\textbf{Claim 1:} Any triple that is not degree-singular has a unique entry of maximum degree.

\emph{Proof of Claim 1.} Assume, for a contradiction, that the degrees of $(A(q),B(q),C(q))$ satisfy $a=k$ and $b=c=t$ for some $k<t$: $A(q)=a_kq^k+\cdots$, $B(q)=b_tq^t+\cdots $, $
    C(q)=c_tq^t+\cdots$. Plugging them into \eqref{eq:mirrorMarkov} we get,

\begin{equation*}
A^2 + B^2 + C^2 + (q+q^{-1})(AB+BC+AC) = 3(1 + q + q^{-1})ABC.
\end{equation*}
The left hand side of the equation has maximum degree $2t+1$, with the coefficient of $q^{2t+1}$ equal to $b_tc_t$. The right hand side has maximum degree $2t+k+1$, and the coefficient of $q^{2t+k+1}$ is $3a_kb_tc_t$. As $b_t$ and $c_t$ are non-zero, this is not possible.

\textbf{Claim 2:} The only degree-singular triple is $(1,1,1)$.

\emph{Proof of Claim 2.} Taking $a=b=c=t$, we set:
\begin{align*}
    A(q)&=a_tq^t+a_{t-1}q^{t-1}+\cdots+a_{t-1}q^{-(t-1)}+a_tq^{-t},\\
    B(q)&=b_tq^t+b_{t-1}q^{t-1}+\cdots+b_{t-1}q^{-(t-1)}+b_tq^{-t},\\
    C(q)&=c_tq^t+c_{t-1}q^{t-1}+\cdots+c_{t-1}q^{-(t-1)}+c_tq^{-t}.\\
\end{align*}
The degree of the left hand side is \textit{at most }$2t+1$, as the coefficient of $q^{2t+1}$, $a_tb_t+b_tc_t+a_tc_t$, may be equal to zero. The degree of the right hand side is \textit{exactly} $3t+1$, as we know the coefficient of $q^{2t+1}$ is non-zero: $3a_tb_tc_t$.

We can conclude that, we need $t=0$ (meaning $A$,$B$,$C$ are constants) and $a_0b_0+b_0c_0+a_0c_0=3a_0b_0c_0$.

As $a_0$, $b_0$ and $c_0$ are non-zero integers, $|a_0b_0c_0|$ is larger than or equal to $|a_0b_0|$, $|b_0c_0|$ and $|a_0c_0|$. The equality can only be realized when all these values are positive and equal to $1$. So $(1,1,1)$ is indeed the only degree-singular triple.

We will finish our proof by arguing that given any solution, we can reach a degree-singular non zero triple in finitely many steps via mutation. As $(1,1,1)$ is the only such triple this mutation uniquely places the given solution on the deformed squared Markov tree. 

For a non-degree singular solution $(A(q),B(q),C(q))$, we have one coordinate with maximum degree, so we can assume without loss of generality that the degrees satisfy $a\leq b < c$. Let us mutate at the coordinate $C(q)$ to get the triple $(A(q),B(q),C'(q))$ where $C'(q)$ is given by (as in Equation \eqref{eq:squared mutation rule}):

\begin{equation*}
\displaystyle C'=\frac{A^2+(q+q^{-1})AB + B^2}{C}.
\end{equation*}

What can we say about the degree of $C'$? The degree of the numerator is bounded above by $\operatorname{max}\{2b,a+b+1\}$. As $c\geq b+1$, both $2b-c$ and $a+b+1-c$ are less than or equal to $b$. So the maximum degree of the new triple $(A(q),B(q),C'(q))$ is strictly smaller. If $(A(q),B(q),C'(q))$ is degree-singular, then we are done. If not, we can repeat mutating at the coordinate with maximum degree. As the progress can not go on indefinitely (more than $c$ steps to be exact), we will reach a degree-singular solution after finitely many mutations. 
\end{proof}

\section{Specializations}~\label{sec:spec}

In this section, we particularly highlight the fact that our deformed squared Markov equation is unifying the generalized Markov equations which have fruitfully dominated the recent interest in this topic. In Subsection~\ref{subsec:super}, we prove an interesting conjecture, Conjecture~\ref{conj:Musiker}, of Musiker.

\subsection{Specialization to integers}\label{sec:spec int}
When $q+q^{-1}$ is specialized to a nonnegative integer $k$, the deformed squared Markov equation \eqref{eq:mirrorMarkov} becomes the $k$-generalized Markov equation ($k$-GM equation) studied in \cite{GMS} by Gyoda, Maruyama and Sato.
\begin{equation}
    \label{eq:kgeneralizedMarkov}
    \tag{\eqMarkk}
    x^2+y^2+z^2 + k(yz+zx+xy) = (3+3k)xyz.
\end{equation}

In \cite{GMS}, the authors give a combinatorial interpretation of the $k$-generalized Markov numbers, solutions to \eqMarkk. We translate here this construction in our $q$-deformed setting and give a proof for Theorem~\ref{thm:positivity}.

Let us recall some notations. Let $a_1,...,a_n$ be integers. We define $[a_1,...,a_n]$ to be the continued fraction:
\begin{align*}
    [a_1,...,a_n]\coloneq a_1 + \cfrac{1}{a_2 + \cfrac{1}{\ddots + \cfrac{\ddots}{a_{n-1} + \cfrac{1}{a_n}}}}
\end{align*}

Let $r=\frac{a}{b},s=\frac{c}{d}$ be rational numbers in lowest terms, ordered such that $r<s$. If $|ad-bc|=1$, the \defin{Farey sum} of the two is defined to be 
\begin{align*}
    t = r\oplus s=\frac{a+c}{b+d}.
\end{align*}

Then $(r,s,t)$ is called a \defin{Farey triple}. If $(r,s,t)$ is a Farey triple, then so is $(r, t, r\oplus t)$ and $(t,s, s\oplus t)$. This allows to define recursively the \defin{Farey tree}, starting from $r=0$ and $s=1$.

Recall that Markov triples are parametrized by rational numbers, via the identification between the Markov tree and the Farey tree, so that for each $t\in \QQ\cap (0,1)$, there is a Markov number $\mathbf{m}_t$ associated to it, starting with $\mathbf{m}_{1/2} = 5$. Then each Farey triple $(r,s,t)$ with $t = r\oplus s$ corresponds to a Markov triple $(\mathbf{m}_r,\mathbf{m}_s,\mathbf{m}_t)$. The $k$-generalized Markov tree and the deformed squared Markov tree are also parametrized by $\QQ\cap(0,1)$ in the same fashion. Gyoda, Maruyama and Sato proved that given a parameter $t\in \QQ\cap(0,1)$, one can recover the $k$-generalized Markov triple associated to it as numerators of some continued fractions defined recursively along the tree. The same process works for deformed squared Markov triples.

\begin{definition}[Proposition 7.12 of \cite{GMS}]
Let $t\in \QQ\cap (0,1)$, and let $(r,s,t)$ be a Farey triple. Define recursively the continued fraction $F^+_t(q)$ by the following rules :
\begin{itemize}
    \item $F^+_{1/2}(q) := [2(q+q^{-1}) + 2,q+q^{-1}+2]$.
    \item If $r = 0$ and $s \neq 1$, and given $F^+_s(q) = [b_1,\cdots,b_n]$, put 
    $F^+_t(q) := [2(q+q^{-1})+2,1,b_n-1,b_{n-1},\cdots,b_1].$
    \item If $r\neq 0$ and $s=1$, and given $F^+_r(q) = [a_1,\cdots,a_{\ell}]$, put 
    $
    F^+_t(q) = [a_{\ell},\cdots,a_1,3(q+q^{-1})+2,q+q^{-1}+2].
    $
    \item If $r\neq 0$ and $s\neq 1$, and given $F^+_s(q) = [b_1,\cdots,b_n]$ and $F^+_r(q) = [a_1,\cdots,a_{\ell}]$, 
    $$
    F^+_t(q) = [a_{\ell},\cdots,a_1,3(q+q^{-1})+2,1,b_{n}-1,b_{n-1},\cdots,b_1].
    $$
\end{itemize}
\end{definition}

Let $\mathbf{m}_t$ be the Markov number associated to $t\in \QQ\cap (0,1)$. Let $M_t(q)$ be the deformed squared of $\mathbf{m}_t$. By Corollary 7.11 in \cite{GMS}, the numerator of $F^+_t(q)$ is precisely $M_t(q)$.

To prove the positivity of coefficients of (the numerator of) $F^+_t(q)$, we will use the following technical lemma.

\begin{lemma}
    For every $t\in \QQ\cap(0,1)$, both the first and the last entry $a_1$ and $a_n$ of the continued fraction $F^+_t(q) = [a_1,a_2\cdots,a_n]$ are Laurent polynomials with constant coefficient greater than $2$. 
\end{lemma}

\begin{proof}
    By induction on the Farey tree.\\
    \noindent As $F^+_{1/2}(q) = [2q+2q^{-1} + 2,q+q^{-1},q+q^{-1}+2]$, the lemma is true for $t = 1/2$.\\
    \noindent Let $(r,s,t)$ be a Farey triple with $t = r\oplus s$. If $r\neq0$ and $s\neq 1$, suppose the lemma is true for $r$ and $s$. Then the continued fraction $F^+_t(q)$ begins by the last entry of $F^+_r(q)$, and ends by the first entry of $F^+_s(q)$. If $r = 0$, $F_t^+(q)$ begins with $2q+2q^{-1} + 2$ and ends with the first entry of $F_s^+(q)$. Similarly if $s=1$, $F^+_t(q)$ begins with the last entry of $F_r^+(q)$ and ends with $q+q^{-1}+2$. In any case, beginnings and ends of $F^+_t(q)$ are Laurent polynomials whose constants coefficients are greater than $2$.
\end{proof}

Now we are ready to prove the following.

\begin{proof}[Proof of Theorem~\ref{thm:positivity}]
    Let $(A,B,C)$ be a triple in the deformed squared Markov tree, parametrized by some rational $t\in \QQ\cap(0,1)$. Via the method described above, one can associate to this triple a continued fraction $F^+_t(q) = [a_1,a_2,\cdots,a_n]$, where the $a_i's$ are Laurent polynomials with positive integer coefficients. The numerator of $F^+_t(q)$ is also a Laurent polynomial with positive integer coefficients. According to Corollary 7.11 in \cite{GMS}, this numerator is precisely the squared Markov number $C(q)$.
\end{proof}

\begin{example}
    Take $t = \frac{1}{3}$. The corresponding Markov number is $13$, and its deformed squared is 
    $$
    M_{1/3}(q) = 4q^{-3} + 18q^{-2} + 38q^{-1} + 49 + 38q + 18q^2 + 4q^3.
    $$ 
    \noindent The Farey triple corresponding to this is $\left(\frac{0}{1},\frac{1}{2},\frac{1}{3}\right)$, so the continued fraction is 
    \begin{align*}
    F^+_{1/3}(q) &= [2(q+q^{-1})+2,1,q+q^{-1}+1,2(q+q^{-1})+2]\\
    &= \frac{4q^6 + 18q^5 + 38q^4 + 49q^3 + 38q^2 + 18q + 4}{2q^5 + 6q^4 + 9q^3 + 6q^2 + 2q}.\\
    \end{align*}
\end{example}

\begin{remark}
\begin{enumerate}
    \item[(i)] The specialization $q=i$ corresponds to $k=0$ in the generalized Markov triples, which gives back the usual Markov numbers. It means that the deformed squared Markov numbers contain the usual Markov numbers in their value at $q =i$.
    \item[(ii)] The specialization $q = \frac{1+i\sqrt{3}}{2}$ at a $6$th root of unity corresponds to $k = 1$ for generalized Markov numbers, and this case was studied in \cite{BS24}.
    \item[(iii)] In \cite{GM23}, it is shown in particular that every $k$-GM triple appears exactly once in the $k$-generalized Markov tree. However, this result fails for $k <0$. For example, at $q = \frac{-1 + i\sqrt{3}}{2}$, we get $k = q + q^{-1} = -1$, and all the squared Markov numbers collapse to $1$, because the mutation rule is degenerate in this case and send $(1,1,1)$ to itself. 
\end{enumerate}
\end{remark}

\subsection{Super Markov numbers}~\label{subsec:super}

In decorated Teichmüller theory, a marked surface is given a hyperbolic metric, where there is a cusp at each marked point. Ideal arcs on the surface have an invariant called the lambda-length, which is dependent on chosen horocycles at the marked points. Under a fixed triangulation of the marked surface, these lambda-lengths define global coordinates on the decorated Teichmüller space \cite{P87}. 

Lambda-lengths were shown to be a set of cluster coordinates for a cluster algebra from marked surfaces \cite{FST08, FT18}. In particular, given a lambda-length coordinate system for a decorated Teichmüller space with a fixed triangulation, the change in triangulation gives rise to another lambda-length coordinate system related by the Ptolemy transformation. 

In \cite{PZ19} the authors define super-Teichmüller space with analogous super lambda-lengths. For an arc on a bordered surface with marked points, Musiker, Ovenhouse, and Zhang gave a matrix formula that yields the arc's super lambda-length \cite{Musiker2023} and show that the super lambda lengths give a structure of a super cluster algebra. We follow the exposition in \cite{MOZ21, M25}. 
Let $P$ be a polygon. The \defin{decorated super-Teichmüller space} of $P$ with oriented triangulation $T$ is a superalgebra with even generators $\lambda_{ij}$ called the \defin{super lambda-lengths} and for each triangle with vertices $i,j,k$, an odd generator $\boxed{ijk}$. When two triangulations are related by the \defin{super Ptolemy transformation} given below, 

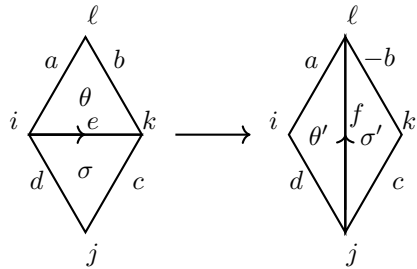
\begin{figure}[h!]
\centering
\begin{tikzpicture}[scale=0.5, baseline, thick]

    \draw (0,0)--(3,0)--(60:3)--cycle;
    \draw (0,0)--(3,0)--(-60:3)--cycle;

    \draw[->, thick] (0,0) -- (1.5,0);

    \draw node[above] at (0.6, 1.5){$a$};
    \draw node[above] at (2.4, 1.5){$b$};
    \draw node[above] at (-0.4, -0.1){$i$};
    \draw node[above] at (3.2, -0.1){$k$};
    \draw node[above] at (0.2, -1.7){$d$};
    \draw node[above] at (2.9, -1.7){$c$};
    \draw node at (1.7, -3.2){$j$};
    \draw node at (1.7, 3.2){$\ell$};
    \draw node at (1.7,0.3){$e$};
    
    \draw node at (1.5,1){$\theta$};
    \draw node at (1.5,-1){$\sigma$};
\end{tikzpicture}
\begin{tikzpicture}[baseline]
    \draw[->, thick](0,0)--(1,0);
    \node[above]  at (0.5,0) {};
\end{tikzpicture}
\begin{tikzpicture}[scale=0.5, baseline, thick,every node/.style={sloped,allow upside down}]
    \draw (0,0)--(60:3)--(-60:3)--cycle;
    \draw (3,0)--(60:3)--(-60:3)--cycle;

    \draw[->, thick] (1.5,-2) -- (1.5,0);

    \draw node[above] at (0.6, 1.5){$a$};
    \draw node[above] at (2.4, 1.5){$-b$};
    \draw node[above] at (-0.4, -0.1){$i$};
    \draw node[above] at (3.2, -0.1){$k$};
    \draw node[above] at (0.2, -1.7){$d$};
    \draw node[above] at (2.9, -1.7){$c$};
    \draw node at (1.7, -3.2){$j$};
    \draw node at (1.7, 3.2){$\ell$};
    \draw node at (1.8,0.5){$f$};
    
    \draw node at (0.8, 0){$\theta'$};
    \draw node at (2.2, 0){$\sigma'$};
\end{tikzpicture}

\caption{Super Ptolemy transformation, where $-$ indicates opposite orientation.}
\label{fig:super_ptolemy}
\end{figure}

\noindent new elements of the superalgebra are defined by the super Ptolemy relations:
\begin{align*}
    ef=ac+bd+\sqrt{abcd}\sigma\theta,\\
    \sigma'=\frac{\sigma\sqrt{bd}-\theta\sqrt{ac}}{\sqrt{ac+bd}}, \\
    \theta'=\frac{\theta\sqrt{bd}+\sigma\sqrt{ac}}{\sqrt{ac+bd}},
\end{align*}
where the order of multiplying $\sigma$ and $\theta$ together depends on the orientation of the edge being flipped. The new cluster variables, $\theta$ and $\sigma$, satisfy that $(\sigma\theta)^2=0$. The chosen orientation of the edges of a triangulation, modulo an equivalence relation given in Figure \ref{fig:equivalence-relation}, forms a \defin{spin structure}.

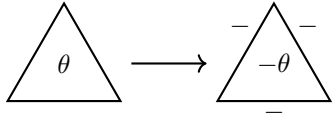
\begin{figure}[h!]
\begin{tikzpicture}[scale=0.5, baseline, thick]

    \draw (0,0)--(3,0)--(60:3)--cycle;
    
    \draw node at (1.5,1){$\theta$};
\end{tikzpicture}
\begin{tikzpicture}[baseline]
    \draw[->, thick](0,0.5)--(1,0.5);
    \node[above]  at (0.5,0) {};
\end{tikzpicture}
\begin{tikzpicture}[scale=0.5, baseline, thick]

    \draw (0,0)--(3,0)--(60:3)--cycle;

    \draw node[above] at (0.6, 1.5){$-$};
    \draw node[above] at (2.4, 1.5){$-$};

    \draw node at (1.5,-0.3){$-$};
    
    \draw node at (1.5,1){$-\theta$};
\end{tikzpicture}

    \caption{Equivalence relation, where $-$ indicates the edge with opposite orientation}
    \label{fig:equivalence-relation}
\end{figure}

In \cite{Huang2023} Huang, Penner, and Zeitlin study the decorated super-Teichmüller space of a once punctured torus. The decorated super-Teichmüller space of a once punctured torus with spin structure corresponding to a triangulation cyclically oriented as pictured in Figure \ref{fig:triangulation-once-punctured-torus} was shown to have its super lambda-lengths satisfy an equation called the \defin{super Markov equation} \cite{Huang2023} (Equation 26):
\[x^2+y^2+z^2+(xy+yz+xz)\sigma\theta =3(1+\sigma\theta)xyz.\]
Solutions to the above equation starting with $(1,1,1)$  are called \defin{super Markov numbers}. The mutation rule for super Markov numbers is given by \cite{M25}:
\[x'=\frac{y^2+yz\sigma\theta+z^2}{x}.\]

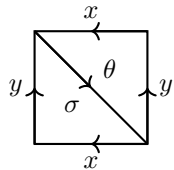
\begin{figure}[h]
    \centering
    \begin{tikzpicture}[scale=0.5, baseline, thick]

    \draw (0,0)--(3,0)--(3,3)--(0,3)--cycle;
    \draw (3,0) -- (0,3);
    \draw[->, thick] (0,3) -- (1.5, 1.5);
    \draw[->, thick] (3,0) -- (3, 1.5);
    \draw[->, thick] (0,0) -- (0, 1.5);
    \draw[->, thick] (3,3) -- (1.5, 3);
    \draw[->, thick] (3,0) -- (1.5, 0);
    
    \draw node at (1,1){$\sigma$};
    \draw node at (2, 2) {$\theta$};

    \draw node at (-0.5, 1.5){$y$};
    \draw node at (3.5, 1.5){$y$};
    \draw node at (1.5, -0.5){$x$};
    \draw node at (1.5, 3.5){$x$};

\end{tikzpicture}
    \caption{Spin structure on the once-punctured torus induced from cyclic orientation of two triangles}
    \label{fig:triangulation-once-punctured-torus}
\end{figure}

In \cite{M25}, Musiker poses the following conjecture (Conjecture 5.1):

\begin{conjecture}~\label{conj:Musiker}
    Assuming the spin-structure on the once-punctured torus induced from the cyclic orientation of the two triangles in the fundamental domain (clockwise on one of them, counter-clockwise on the other), the formulas for super lambda-lengths of arcs in the associated decorated super-Teichmüller space can all be expressed using only positive coefficients.
\end{conjecture}

When we specialize the deformed squared Markov equation we get the super Markov numbers. Therefore, Proposition \ref{thm:positivity} proves this conjecture.

\begin{corollary}
    Conjecture 5.1. in~\cite{M25} is true assuming we start with an initial triangulation where all three super lambda-lengths equal 1.
\end{corollary}

\begin{proof}
     Specializing at $q+q^{-1}=\varepsilon$ in the deformed squared Markov equation, we get the super Markov equation, where $\varepsilon=\sigma\theta$. If we now consider $F^+(\varepsilon)$, we obtain the positivity result for the super Markov equation in~\cite{M25} which proves the conjecture. 
\end{proof}

Now that this conjecture is proven to be true, further questions in \cite[Question 5.2, 5.3 and 5.4]{M25} can be further investigated. 

\section{Mirror deformation}~\label{sec:mirrorMarkov}

As we mentioned in the abstract, symmetric solutions of the deformed squared Markov equation factorize nicely. Now, in this section we consider `\emph{square roots}' of these solutions. We completely analyse them and prove a natural and new $q$-deformation of the Markov numbers along with developing the new combinatorics.

\subsection{The mirror Markov tree}

We will see that the symmetric Laurent polynomials that occur in the solutions of the quantized Equation~\eqref{eq:mirrorMarkov} are all of form $A(q)=a(q)a(q^{-1})$ where $a(q)$ is a polynomial. Up to multiplying with a power of $q$, $A(q)$ can be viewed as the product of two polynomials that are mirror images of each other. In this section, we will describe a mutation algorithm giving the polynomials $a(q)$ directly, developing a new $q$-deformation for Markov numbers, which we call the \defin{mirror deformation}.

We define the \defin{mirror Markov tree} as the tree when we consider the $q$-factors $x(q)$'s in $X(q)$ of the deformed squared Markov tree; so vertices of the mirror Markov tree are these new $q$-deformation for Markov numbers.

We will order our triples $(a(q),b(q),c(q))$ to satisfy $a(1)\leq b(1) \leq c(1)$. The first few mirror deformations are given below on Figure~\ref{fig:initialization}; colours of the edges will be explained later.

\begin{figure}[h!]
\centering
\begin{tikzcd}[column sep=tiny]
& {(1,1,1) } \arrow[d, no head,blue] & \\
& {(1,1,1+q) } \arrow[d, no head,blue] & \\
& {(1,1+q,1+2q+2q^2)} \arrow[ld, no head,blue] \arrow[rd, no head, red] & \\
{(1,1+2q+2q^2,2+4q+5q^2+2q^3)} & & {(1+q,1+2q+2q^2,3 + 8q + 10q^2 + 6q^3+2q^4)}
\end{tikzcd}
\caption{Initial steps of mirror Markov tree}
\label{fig:initialization}
\end{figure}
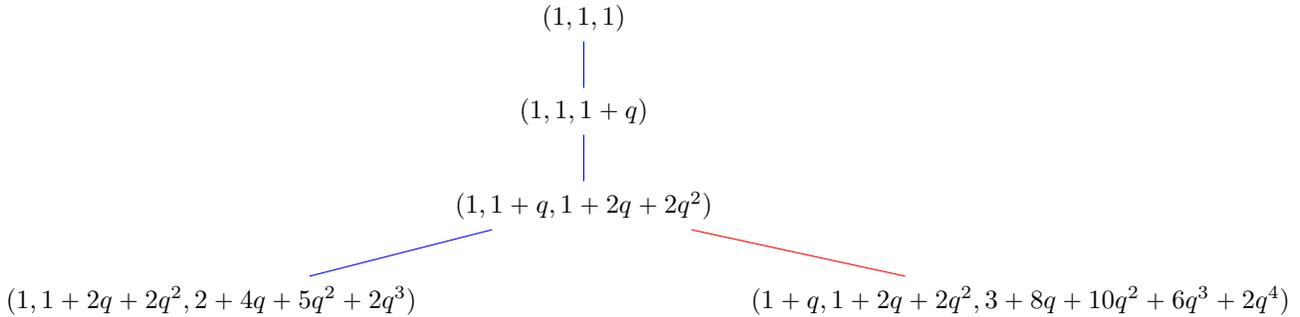

We want to remark that there is some choice involved at this point. For the Markov number $5$, the deformation of $5^2$ can be factorized as $(2+2q+q^2)(2+2q^{-1}+q^{-2})$ as easily as $(1+2q+2q^2)(1+2q^{-1}+2q^{-2})$, giving a second candidate $(2+2q+q^2)$. This duality of choices exists throughout and allows us some degree of freedom, but not complete freedom if we want a well-defined lift of the Markov mutation rule to the $q$-version. See Remark~\ref{rmk:choice}
for more discussion on this subject.

Consider a triple $t=(a(q),b(q),c(q))$ in the mirror Markov tree. Recall that we denote by $\widetilde{p}(q)$ the mirror polynomial of $p(q)$. We define a mutation rule as follows.

\begin{definition}[Mutation Rule]\label{def:positive_negative_mut}
$$
\mu_\pm(a(q),b(q),c(q))\rightarrow\begin{cases}
\left(b(q),c(q),q^{\deg(c)}\cfrac{q^\pm B(q)+C(q)}{\tilde{a}(q)}\right) \quad \text{(right-edge)}\\[15pt]
\left(a(q),c(q),q^{\deg(c)}\cfrac{q^\pm A(q)+C(q)}{\tilde{b}(q)}\right) \quad \text{(left-edge)}\\
\end{cases}
$$
\end{definition}

Mutation at $a(q)$ replaces $a(q)$ with $q^{\deg(c)}(q^{\pm 1}B(q) + C(q))/\tilde{a}(q)$ where the mutation is called a positive mutation (denoted $\mu_+$) or a negative mutation (denoted $\mu_-)$ to match the sign of the power of $q$.

Similarly, mutation at $b(q)$ replaces $b(q)$ with $(q^{\pm 1 }A(q) + C(q))/\tilde{b}(q)$ where the mutation is again called a positive mutation (denoted $\mu_+$) or a negative mutation (denoted $\mu_-)$ to match the sign of the power of $q$.

We visualize the mutation operation by using an edge going to the right for a mutation at $a(q)$ (it will be called a \defin{right-edge}) and an edge going to the left represents a mutation at $b(q)$ (it will be called a \defin{left-edge}). 

We use the colour $blue$ to depict edges with positive mutation and $red$ to depict edges with negative mutation. Now, the question is how we decide on whether to perform a positive or negative mutation rule at a triple. We make this precise with the following Definition~\ref{def:mutationrule}.

\begin{definition}[Sign of the Mutation Rule] 
\label{def:mutationrule}
Consider a triple $t$ in the mirror Markov tree. Denote by $t'$ the parent of $t$ and by $\epsilon$ the sign of the mutation $t' \rightarrow t$. Then the signs of the mutations from $t$ are defined as follows

\textbf{I.} The left-edge mutation from $t$ has sign $\epsilon$.

\textbf{II.} The right-edge mutation from $t$ has sign $-\epsilon$.
\end{definition}

See Figure~\ref{fig:mutationrules} for a visual depiction of the rules.

\begin{figure}[H]
    \centering
    \begin{tikzpicture}
    \begin{scope}[scale=0.9]
        \draw[line width=1pt,blue] (0,1)--(0,0);
        \node[blue] at (0.2,0.5) {$\epsilon$};
        \draw[line width=1pt,blue] (0,0)--(-1,-1);
        \node[blue] at (-0.2,-0.5) {$\epsilon$};
        \draw[line width=1pt,red] (0,0)--(1,-1);
        \node[red] at (1,-0.5) {$-\epsilon$};
    \end{scope}

    \begin{scope}[scale=0.7,shift={(5,0)}]
        \draw[line width=1pt,red] (0,1)--(0,0);
        \node[red] at (0.2,0.5) {$\epsilon$};
        \draw[line width=1pt,red] (0,0)--(-1,-1);
        \node[red] at (-0.2,-0.5) {$\epsilon$};
        \draw[line width=1pt,blue] (0,0)--(1,-1);
        \node[blue] at (1,-0.5) {$-\epsilon$};
    \end{scope}
\end{tikzpicture}
\caption{Mutation rules}
\label{fig:mutationrules}
\end{figure}
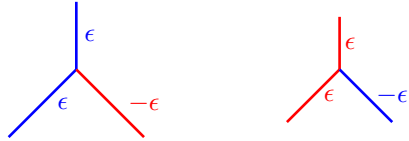

We prove in the following Theorem~\ref{thm:mutation} that this choice of mutation rule ensures that the division gives an actual polynomial, i.e. we have a well-defined mutation. Note that this fact does not automatically follow from the definition of mutation. For example, in the left-edge mutation from $(1,5,13)_q$, the negative mutation gives

$$
\mu_-(1,1+2q+2q^2, 2+4q+5q^2+ 2q^3) = \frac{4+18q+39q^2+49q^3+38q^4 + 18q^5 + 4q^6}{2+2q+q^2},
$$
which is not a polynomial in $q$.

\begin{remark}~\label{rem:initial_choice}
\begin{enumerate}
    \item[(i)]~\label{rmk:choice} Note that the proof of the recursion rule does not depend on the initial choices we made. Different initial choices result in trees with similar recursive structure, where colours are shifted and some deformations are replaced with their mirror images, with the resulting combinatorics remaining unchanged. Our particular choice favours the mutation $\mu_+$ at the junctures where both mutations give polynomials.
    \item[(ii)] The mutation rule in Definition \ref{def:mutationrule} may seem complicated with this division by the mirror image of the previous polynomial. Actually, this rule was chosen so that the sign patterns in the mirror Markov tree is as simple as possible. In particular, if one removes the tilde in the mutation rule, then the sign pattern would have recursive rules of order two instead of one. 
\end{enumerate}

\end{remark}

In Figure \ref{fig:markov_tree3} below, are the first levels of the tree, and in Figure \ref{fig:colour_pattern} is the general colour pattern of the tree. 

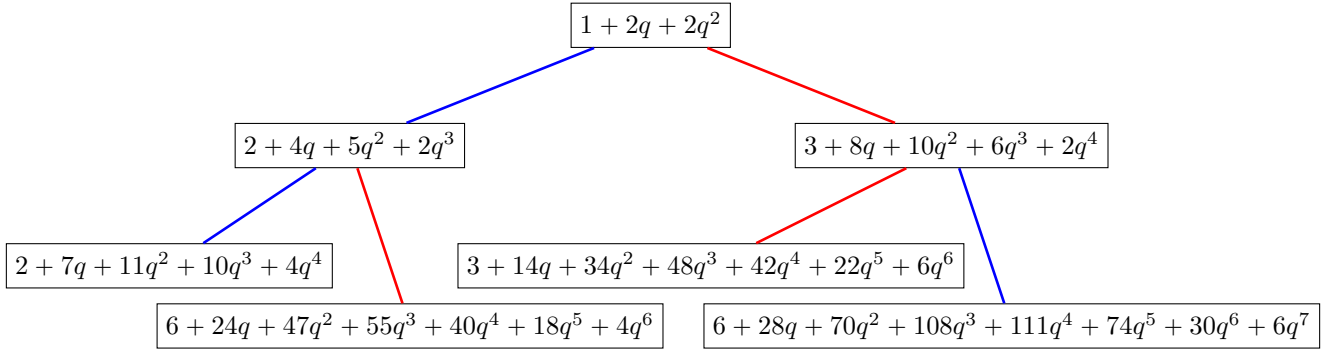
\begin{figure}[h]
\begin{tikzpicture}[scale=0.8]
\node[draw,rectangle] (A) at(0,0) {$1+2q+2q^2$};\node[draw,rectangle] (C) at(5,-2) {$3 + 8q + 10q^{2} + 6q^{3} + 2q^{4}$} ;
\node[draw,rectangle] (B) at(-5,-2) {$2+4q+5q^2 + 2q^3$} ;
\node[draw,rectangle] (D) at(-8,-4) {$2 + 7q + 11q^2 + 10q^3 + 4q^4$} ;
\node[draw,rectangle] (E) at(-4,-5) {$6 + 24q+ 47q^2 + 55q^3 + 40q^4 + 18q^5 + 4q^6$} ;
\node[draw,rectangle] (F) at(1,-4) {$3 + 14q + 34q^2 + 48q^3 + 42q^4 + 22q^5 + 6q^6$} ;
\node[draw,rectangle] (G) at(6,-5) {$6 + 28q + 70q^2 + 108q^3 + 111q^4 + 74q^5 + 30q^6 + 6q^7$} ;

\draw[line width=1pt, blue] (A)-- (B);
\draw[line width=1pt, red] (A)-- (C);
\draw[line width=1pt,blue] (B)-- (D);
\draw[line width=1pt, red] (B)-- (E);
\draw[line width=1pt,red] (C)-- (F);
\draw[line width=1pt,blue] (C)-- (G);
\end{tikzpicture}
\caption{Mirror deformed Markov numbers}
\label{fig:markov_tree3}
\end{figure}

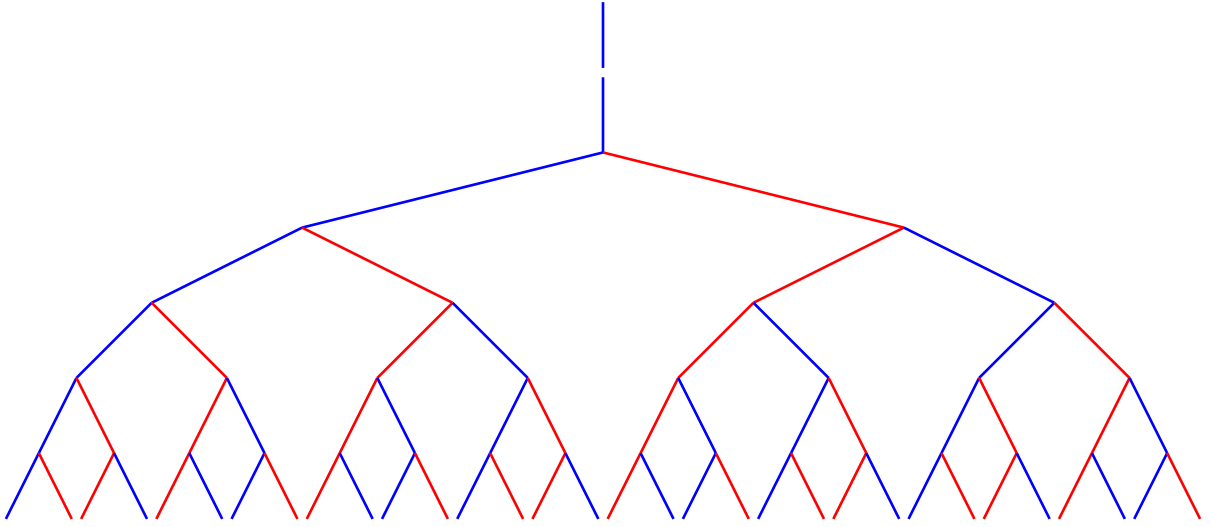
\begin{figure}[H]
    \centering
    \begin{tikzpicture}
            \node (A0) at(0,2) {};
            \node (A1) at(0,1) {};

            \node (A) at(0,0) {};
            
            \node (B1) at(-4,-1) {};
            \node (B2) at(4,-1) {};
            
            \node (C1) at(-6,-2) {};
            \node (C2) at(-2,-2) {};
            \node (C3) at(2,-2) {};
            \node (C4) at(6,-2) {};
            
            \node (D1) at(-7,-3) {};
            \node (D2) at(-5,-3) {};
            \node (D3) at(-3,-3) {};
            \node (D4) at(-1,-3) {};
            \node (D5) at(1,-3) {};
            \node (D6) at(3,-3) {};
            \node (D7) at(5,-3) {};
            \node (D8) at(7,-3) {};
            
            \node (E1) at(-7.5,-4) {};
            \node (E2) at(-6.5,-4) {};
            
            \node (E3) at(-5.5,-4) {};
            \node (E4) at(-4.5,-4) {};

            \node (E5) at(-3.5,-4) {};
            \node (E6) at(-2.5,-4) {};
            
            \node (E7) at(-1.5,-4) {};
            \node (E8) at(-0.5,-4) {};

            \node (E9) at(0.5,-4) {};
            \node (E10) at(1.5,-4) {};

            \node (E11) at(2.5,-4) {};
            \node (E12) at(3.5,-4) {};
            
            \node (E13) at(4.5,-4) {};
            \node (E14) at(5.5,-4) {};

            \node (E15) at(6.5,-4) {};
            \node (E16) at(7.5,-4) {};
            
            \node (F1) at(-8,-5) {};
            \node (F2) at(-7,-5) {};
            
            \node (F3) at(-7,-5) {};
            \node (F4) at(-6,-5) {};
            
            \node (F5) at(-6,-5) {};
            \node (F6) at(-5,-5) {};
            
            \node (F7) at(-5,-5) {};
            \node (F8) at(-4,-5) {};

            \node (F9) at(-4,-5) {};
            \node (F10) at(-3,-5) {};

            \node (F11) at(-3,-5) {};
            \node (F12) at(-2,-5) {};

            \node (F13) at(-2,-5) {};
            \node (F14) at(-1,-5) {};

            \node (F15) at(-1,-5) {};
            \node (F16) at(0,-5) {};

            \node (F17) at(0,-5) {};
            \node (F18) at(1,-5) {};

            \node (F19) at(1,-5) {};
            \node (F20) at(2,-5) {};

            \node (F21) at(2,-5) {};
            \node (F22) at(3,-5) {};

            \node (F23) at(3,-5) {};
            \node (F24) at(4,-5) {};

            \node (F25) at(4,-5) {};
            \node (F26) at(5,-5) {};

            \node (F27) at(5,-5) {};
            \node (F28) at(6,-5) {};

            \node (F29) at(6,-5) {};
            \node (F30) at(7,-5) {};

            \node (F31) at(7,-5) {};
            \node (F32) at(8,-5) {};

            \draw[line width=1pt,blue] (A0.center)--(A1);
            \draw[line width=1pt,blue] (A1.center)--(A.center);

            \draw[line width=1pt,blue] (A.center)-- (B1.center);
            \draw[line width=1pt,red] (A.center)-- (B2.center);
            
            \draw[line width=1pt,blue] (B1.center)-- (C1.center);
            \draw[line width=1pt,red] (B1.center)-- (C2.center);
            \draw[line width=1pt,red] (B2.center)-- (C3.center);
            \draw[line width=1pt,blue] (B2.center)-- (C4.center);
            
            \draw[line width=1pt,blue] (C1.center)-- (D1.center);
            \draw[line width=1pt,red] (C1.center)-- (D2.center);
            \draw[line width=1pt,red] (C2.center)-- (D3.center);
            \draw[line width=1pt,blue] (C2.center)-- (D4.center);
            \draw[line width=1pt,red] (C3.center)-- (D5.center);
            \draw[line width=1pt,blue] (C3.center)-- (D6.center);
            \draw[line width=1pt,blue] (C4.center)-- (D7.center);
            \draw[line width=1pt,red] (C4.center)-- (D8.center);
            
            \draw[line width=1pt,blue] (D1.center)-- (E1.center);
            \draw[line width=1pt,red] (D1.center)-- (E2.center);

            \draw[line width=1pt,red] (D2.center)-- (E3.center);
            \draw[line width=1pt,blue] (D2.center)-- (E4.center);
            
            \draw[line width=1pt,red] (D3.center)-- (E5.center);
            \draw[line width=1pt,blue] (D3.center)-- (E6.center);

            \draw[line width=1pt,blue] (D4.center)-- (E7.center);
            \draw[line width=1pt,red] (D4.center)-- (E8.center);
            
            \draw[line width=1pt,red] (D5.center)-- (E9.center);
            \draw[line width=1pt,blue] (D5.center)-- (E10.center);

            \draw[line width=1pt,blue] (D6.center)-- (E11.center);
            \draw[line width=1pt,red] (D6.center)-- (E12.center);
            
            \draw[line width=1pt,blue] (D7.center)-- (E13.center);
            \draw[line width=1pt,red] (D7.center)-- (E14.center);

            \draw[line width=1pt,red] (D8.center)-- (E15.center);
            \draw[line width=1pt,blue] (D8.center)-- (E16.center);

            \draw[line width=1pt,blue] (E1.center)-- (F1);
            \draw[line width=1pt,red] (E1.center)-- (F2);
            
            \draw[line width=1pt,red] (E2.center)-- (F3);
            \draw[line width=1pt,blue] (E2.center)-- (F4);
            
            \draw[line width=1pt,red] (E3.center)-- (F5);
            \draw[line width=1pt,blue] (E3.center)-- (F6);
            
            \draw[line width=1pt,blue] (E4.center)-- (F7);
            \draw[line width=1pt,red] (E4.center)-- (F8);
            
            \draw[line width=1pt,red] (E5.center)-- (F9);
            \draw[line width=1pt,blue] (E5.center)-- (F10);

            \draw[line width=1pt,blue] (E6.center)-- (F11);
            \draw[line width=1pt,red] (E6.center)-- (F12);

            \draw[line width=1pt,blue] (E7.center)-- (F13);
            \draw[line width=1pt,red] (E7.center)-- (F14);

            \draw[line width=1pt,red] (E8.center)-- (F15);
            \draw[line width=1pt,blue] (E8.center)-- (F16);

            \draw[line width=1pt,red] (E9.center)-- (F17);
            \draw[line width=1pt,blue] (E9.center)-- (F18);

            \draw[line width=1pt,blue] (E10.center)-- (F19);
            \draw[line width=1pt,red] (E10.center)-- (F20);

            \draw[line width=1pt,blue] (E11.center)-- (F21);
            \draw[line width=1pt,red] (E11.center)-- (F22);

            \draw[line width=1pt,red] (E12.center)-- (F23);
            \draw[line width=1pt,blue] (E12.center)-- (F24);

            \draw[line width=1pt,blue] (E13.center)-- (F25);
            \draw[line width=1pt,red] (E13.center)-- (F26);

            \draw[line width=1pt,red] (E14.center)-- (F27);
            \draw[line width=1pt,blue] (E14.center)-- (F28);

            \draw[line width=1pt,red] (E15.center)-- (F29);
            \draw[line width=1pt,blue] (E15.center)-- (F30);

            \draw[line width=1pt,blue] (E16.center)-- (F31);
            \draw[line width=1pt,red] (E16.center)-- (F32);
\end{tikzpicture}
\caption{colour pattern in the tree}
\label{fig:colour_pattern}
\end{figure}

\vspace{1mm}

\begin{theorem}~\label{thm:mutation} 
With the initial segment and the mutation rule from Definition~\ref{def:mutationrule}, every term that appears in the mirror deformed Markov tree is a polynomial in $q$.
\end{theorem}

\begin{proof}
We keep the same notations as in Definition~\ref{def:mutationrule}. Let us set $t = (a_1,a_2,a_3)$.\\
\textbf{I.} Let us consider the left-edge going from $t$. By the left mutation rule, the sign of this mutation is the same as the mutation $t'\rightarrow t$, and it creates the new variable 
$$
a_0 := q^{\deg(a_3)}\frac{(q^{\epsilon}A_1 + A_3)}{\widetilde{a_2}}.
$$

\noindent As $a_3$ comes from the mutation of some polynomial $a_4$ in the triple $t'$, we can write 
\begin{align*}
q^{\epsilon}A_1 + A_3 &= q^{\epsilon}A_1 + \frac{A_1^2 + (q+q^{-1})A_1A_2 + A_2^2}{A_4} \\
& = \frac{A_2}{A_4}\left(A_2 + (q+q^{-1})A_1\right) + \frac{A_1}{A_4}\left( q^{\epsilon}A_4 + A_1\right)\\
\end{align*}

\begin{center}
\begin{tikzpicture}[scale=1]
\node (B) at(0,-1.5) {$t' = (a_4,a_1,a_2) \text{ or } (a_1,a_4,a_2)$};
\node (C) at(0,-3) {$t = (a_1,a_2,a_3)$};
\node (D) at(-1,-4.5) {$(a_1,a_3,a_0)$};

\draw[line width=1pt] (B)-- (C) node[midway,left] {$\epsilon$};
\draw[line width=1pt] (C)-- (D) node[midway,above] {$\epsilon$};
\end{tikzpicture}
\end{center}

\noindent \underline{Case 1} : the mutation $t' \rightarrow t$ is a left-edge. Then $t' = (a_1,a_4,a_2)$. Let $t''$ be the parent of $t'$, such that $a_2$ comes from a mutation at some $a_5$ in $t''$, of sign $\epsilon$: 
$$
a_2 = q^{\deg(a_4)}\frac{q^{\epsilon}A_1 + A_4}{\widetilde{a_5}}.
$$
\noindent Thus we have $q^{\epsilon}A_4 + A_1 = q^{-\deg(a_4)-\epsilon}a_5\widetilde{a_2}$, so $\widetilde{a_2}$ divides $q^{\epsilon}A_1 + A_3$, and then $a_0$ is a polynomial.
\vspace{12pt}

\noindent \underline{Case 2} : the mutation $t' \rightarrow t$ is a right-edge. Then $t' = (a_4,a_1,a_2)$, and $a_2$ comes from a mutation of sign $ -\epsilon$: 
$$
a_2 = q^{\deg(a_1)}\frac{q^{-\epsilon}A_4 + A_1}{\widetilde{a_5}}.
$$
\noindent Thus we have $\widetilde{a_2}a_5 = q^{\deg(a_1)}(q^{\epsilon}A_4 + A_1)$, and then $a_0$ is a polynomial.
\vspace{12pt}

\noindent \textbf{II.} Let us now consider the right-edge going from $t$. According to the right mutation rule, it creates a new variable with sign $-\epsilon$, 
$$
a_0 = q^{\deg(a_3)}\frac{q^{-\epsilon}A_2 + A_3}{\widetilde{a_1}}.
$$
\noindent The polynomial $a_3$ still comes from a mutation at some $a_4$ in the triple $t'$, so that 
\begin{align*}
q^{-\epsilon}A_2 + A_3 &= q^{-\epsilon}A_2 + \frac{A_1^2 + (q+q^{-1})A_1A_2 + A_2^2}{A_4} \\
& = \frac{A_1}{A_4}\left(A_1 + (q+q^{-1})A_2\right) + \frac{A_2}{A_4}\left( q^{-\epsilon}A_4 + A_2\right)\\
\end{align*}

\noindent The sibling mutation going from $t'$ has sign $-\epsilon$ and gives the triple $s = (a_4,a_2,a_5)$, with 
$$
a_5 = q^{\deg(a_2)}\frac{q^{-\epsilon}A_4 + A_2}{\widetilde{a_1}}.
$$

\noindent Then 
$$
q^{-\epsilon}A_2 + A_3 = \frac{A_1}{A_4}\left(A_1 + (q+q^{-1})A_2\right) + q^{-\deg(a_2)}\frac{A_2}{A_4}a_5\widetilde{a_1},
$$
\noindent so $\widetilde{a_1}$ divides $q^{-\epsilon}A_2 + A_3$, and $a_0$ is a polynomial.
\end{proof}

\begin{notation}~\label{not}
    For $n\geq 1$ a Markov number, let us denote by $m_n(q)\in \ZZ[q]$ the corresponding mirror deformed Markov number, obtained on the mirror Markov tree, and $M_n(q)\in \ZZ[q,q^{-1}]$ the corresponding value on the deformed squared Markov tree. Note that, for all $n$, 
\[M_n(q)=m_n(q)m_n(q^{-1}).\] 
\end{notation}

Here is a table giving the first few values of $m_n$ and $M_n$:
\[
\begin{array}{|c|c|c|}
\hline
    n & m_n & M_n \\
\hline
    1 & 1 & 1 \\[4pt]
    2 & 1+q & q^{-1} + 2 + q \\[4pt]
    5 & 1+2q+2q^2 & 2q^{-2} + 6q^{-1} + 9 + 6q + 2q^2 \\[4pt]
    13 & 2 + 4q + 5q^2 + 2q^3 & 4q^{-3} + 18q^{-2} + 38q^{-1} + 49 + 38q + 18q^2 + 4q^3\\[4pt]
    29 & 3 + 8q + 10q^{2} + 6q^{3} + 2q^{4} & \begin{array}{@{}c@{}}
         6q^{-4} + 34q^{-3} + 98q^{-2} + 176q^{-1}+ 213  \\ + 176q + 98q^2 + 34q^3 + 6q^4
    \end{array}
    \\
    \hline
\end{array}
\]

\subsection{Geometric model for mirror mutation}\label{subsect:geometry}

In this section, we give a heuristic interpretation of our deformation, using tagged triangulations on orbifold surfaces. 

Geometrically, it is known that the Markov cluster algebra arises from a once-punctured torus. For the generalized Markov equations of the type \eqref{eq:2Markov}, or more generally the ones that appear in \cite{GM23}, solutions have a structure of a \defin{generalized cluster algebra}. In \cite{CS14}, Chekhov and Shapiro gave a surface model for these types of generalized cluster algebras, using \defin{orbifold points}. A generalized exchange relation of the form
\begin{equation}\label{eq:gen_clust_mut}
    aa'=b^2 + 2\cos(\pi/p) bc + c^2
\end{equation}
corresponds to the mutation of an arc going around an orbifold point of order $p$, see Figure \ref{fig:mut_orbifold}.

\begin{figure}[H]
    \centering
    \begin{tikzpicture}
    \node at (0,0) {\begin{tikzpicture}[scale=0.8]
    \draw[very thick, red] (0,1.5) to[out=-130,in=180] (0,-0.5) to[out=0,in=-50] (0,1.5);
    \draw[very thick] (0,1.5) to[out=-150, in=150 ] (0,-1.5) to[out=30,in=-30] (0,1.5);
    \draw node at (0,0) {$\bigtimes$};
    \draw node at (0,1.5) {$\bullet$};
    \draw node at (0,-1.5) {$\bullet$};
    \draw[red] node at (0,-0.8) {$a$};
    \draw node at (-1,0) {$b$};
    \draw node at (1,0) {$c$};
\end{tikzpicture}};
\node at (1.5,0) {$\xrightarrow{\mu_a}$};
\node at (3,0) {\begin{tikzpicture}[scale=0.8]
    \draw[very thick, red] (0,-1.5) to[out=130,in=180] (0,0.5) to[out=0,in=50] (0,-1.5);
    \draw[very thick] (0,1.5) to[out=-150, in=150 ] (0,-1.5) to[out=30,in=-30] (0,1.5);
    \draw node at (0,0) {$\bigtimes$};
    \draw node at (0,1.5) {$\bullet$};
    \draw node at (0,-1.5) {$\bullet$};
    \draw[red] node at (0,0.8) {$a'$};
    \draw node at (-1,0) {$b$};
    \draw node at (1,0) {$c$};
\end{tikzpicture}};
\end{tikzpicture}
    \caption{Mutation of an arc around an orbifold point}
    \label{fig:mut_orbifold}
\end{figure}
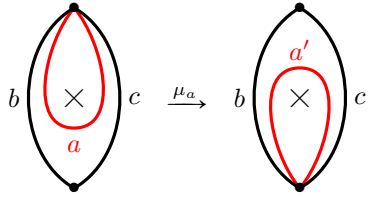

In the case of the deformed squared Markov equation, we have three generalized cluster variables, and each of them mutate via the following generalized exchange relation (see \eqref{eq:squared mutation rule}):
\begin{equation}\label{eq:exchange:XYZ}
    XX'=Y^2 + (q+q^{-1})YZ + Z^2.
\end{equation}
Thus the exchange relations are of the form \eqref{eq:gen_clust_mut}, with $q=e^{i\pi/p}$. 

In \cite{G22} and \cite{BS24}, the authors considered the following generalization of the Markov equation: $x^2 + y^2 +z^2 + xy +xz+yz=6xyz$ (a particular case of \cite{GM23}), for which the surface model is a sphere with one puncture and three orbifold points. Its triangulation is formed by three loops from the puncture, going around the orbifold point (called \emph{pending arcs}).
Our deformed squared Markov equation is a ($q$-)deformation of this generalized Markov equation, where the deformed parameter corresponds to the orders of the orbifold points. Hence, we can use the same surface model of a sphere with one puncture point and three orbifold points, except that the orders of the orbifolds are parameters, see Figure \ref{fig:mut_sphere}. 

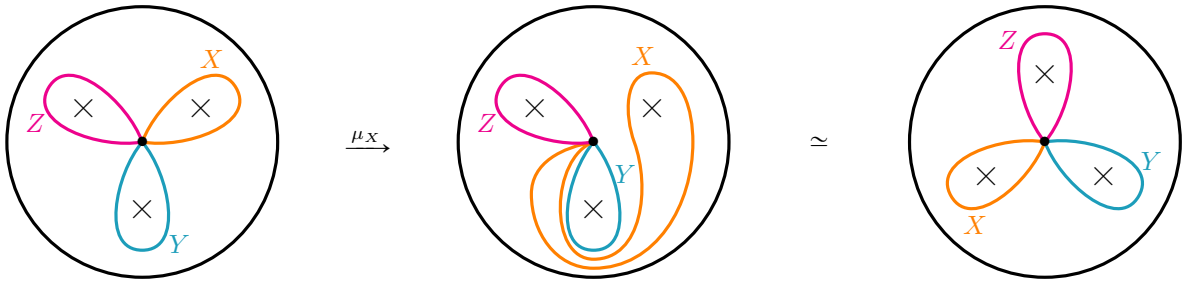
\begin{figure}[H]
    \centering
    \input{mutation_orbifold.tikz}
    \caption{Mutation of an arc in the once-punctured sphere with three orbifold points}
    \label{fig:mut_sphere}
\end{figure}

As noted in \cite{BS24}, according to the mutation rule in Figure \ref{fig:mut_orbifold}, the only effect of a mutation of one pending arc on this surface is to change the cyclic ordering of the arcs, as seen in Figure \ref{fig:mut_sphere}.

Let us assign our triples of deformed squared Markov numbers to triangulations of this surface. First, assume that the initial triangulation is $T_0=\{X,Y,Z\}$, with the loops $X,Y$ and $Z$ in clockwise order around the puncture (as in the left hand side of Figure \ref{fig:mut_sphere}). We associate that triangulation to the initial triple $(1,1,1)$. Then, if we apply any mutation we get to the next triple $(1,1,q^{-1}+2+q)$. On arcs, any mutation of $X,Y$ or $Z$ gives a triangulation homotopic to the right hand side of Figure \ref{fig:mut_sphere}. If we continue this process, along the deformed squared Markov tree, we can associate triangulations of the sphere which are homotopic to alternate cyclic orientations of the three pending arcs.

\begin{remark}
    In the classical setting ($q=1$), this assignment has a precise meaning, as it gives the geodesic lengths of the arcs in the triangulations. In our $q$-deformation, these could be interpreted as some \emph{quantum geodesic lengths}.
\end{remark}

We now consider the mirror deformation. Similarly to the work of Felikson, Shapiro and Tumarkin (see \cite[Definition 5.5]{FSM12}), for a pending arc $X$ going around such an orbifold point, we can replace the orbifold point by a puncture and the arc $X$ by a pair of arcs $x, \tilde{x}$ between the two punctures, one tagged plain and one tagged notched, as in Figure \ref{fig:pending arc to tagged arcs}. As in the cluster algebra interpretation of this replacement in previous work, we write $X=x\tilde{x}$, and assume that both arcs should always be flipped simultaneously. 

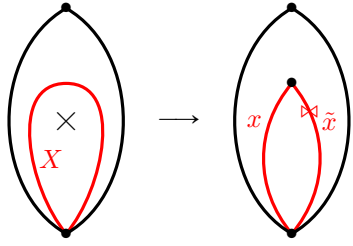
\begin{figure}[H]
    \centering
    \begin{tikzpicture}
    \node at (0,0) {\begin{tikzpicture}
    \draw[very thick, red] (0,-1.5) to[out=130,in=180] (0,0.5) to[out=0,in=50] (0,-1.5);
    \draw[very thick] (0,1.5) to[out=-150, in=150 ] (0,-1.5) to[out=30,in=-30] (0,1.5);
    \draw node at (0,0) {$\bigtimes$};
    \draw node at (0,1.5) {$\bullet$};
    \draw node at (0,-1.5) {$\bullet$};
    \draw[red] node at (-0.2,-0.5) {$X$};
    \end{tikzpicture}};
\node at (1.5,0) {$\longrightarrow$};
\node at (3,0) {\begin{tikzpicture}
    \draw[very thick, red] (0,-1.5) to[out=130,in=-130] (0,0.5) ;
    \draw[very thick,red,postaction={decorate,decoration={markings,
      mark=at position 0.2 with {\node[sloped, allow upside down] {$\bowtie$};}}}] (0,0.5) to[out=-50,in=50] (0,-1.5);
    \draw[very thick] (0,1.5) to[out=-150, in=150 ] (0,-1.5) to[out=30,in=-30] (0,1.5);
    \draw node at (0,0.5) {$\bullet$};
    \draw node at (0,1.5) {$\bullet$};
    \draw node at (0,-1.5) {$\bullet$};
    \draw[red] node at (-0.5,0) {$x$};
    \draw[red] node at (0.5,0) {$\tilde{x}$};
    \end{tikzpicture}};
\end{tikzpicture}
    \caption{Replace a pending arc by a pair of tagged arcs}
    \label{fig:pending arc to tagged arcs}
\end{figure}

In this context, we write the exchange relation \eqref{eq:exchange:XYZ} as
\[
XX'= (q^{1/2}Y + q^{-1/2}Z)(q^{-1/2}Y + q^{1/2}Z).
\]
If we decompose $X$, $Y$ and $Z$ as above as $X=x\tilde{x}$, $Y=y\tilde{y}$ and $Z=z\tilde{z}$, we can write
\[
x'\tilde{x}'=\frac{(q^{1/2}Y + q^{-1/2}Z)(q^{-1/2}Y + q^{1/2}Z)}{x\tilde{x}},
\]
and similarly for $y'\tilde{y}'$ and $ z'\tilde{z}'$.

Let us fix the following mutation rule: 
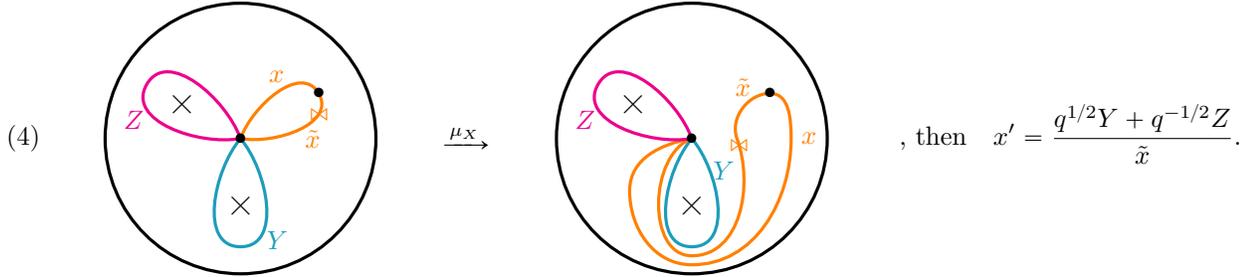
\begin{equation}\label{mutation_rule_orbifold}
    \vcenter{\hbox{\scalebox{1}{\input{mutation_pair_arcs.tikz}}} }\qquad , \, \text{then} \quad x'=\frac{q^{1/2}Y + q^{-1/2}Z}{\tilde{x}}.
\end{equation}
And similarly, $\tilde{x}'=\frac{q^{-1/2}Y + q^{1/2}Z}{x}$. Another way of writing this mutation rule is: for $x$, \emph{the coefficient $q^{1/2}$ is carried by the next loop, in the clockwise order, and the coefficient $q^{-1/2}$ by the previous loop, in the clockwise order}.
Note that the mutation rule only depends on the relative positions of the loops $X,Y,Z$.

\begin{remark}
    \begin{enumerate}
        \item[(i)] The values of $x$ and $\tilde{x}$ obtained here only differs from the one above by a power of $q^{1/2}$. Indeed, $x$, $\tilde{x}$ are defined here symmetrically with respect to $q$: $X=x\tilde{x}$, instead of $X=q^{-\deg x}x\tilde{x}$ as above, and the mutation rule corresponds to Definition \ref{def:positive_negative_mut} up to a power of $q^{1/2}$ as well (see below). That is because the these symmetric definitions and mutation rules seems more natural using the geometric model.
        \item[(ii)] Regarding Remark~\ref{rem:initial_choice} (1), there was also a choice of initial triangulation here ($X,Y,Z$ in clockwise order), the opposite choice would have given a similar pattern, with $x$ and $\tilde{x}$ possibly exchanged. There is again a choice at the second step (see the proof of Proposition~\ref{prop:mutation_orbifold}) that can give either our choice or the opposite choice. 
        \item[(iii)] Regarding Remark~\ref{rem:initial_choice} (2), it is known that for tagged arcs, the mutation changes the tagging, hence $x$ and $\tilde{x}$ are exchanged during the mutation. This explains why the mutation rule for $x$ has an $\tilde{x}$ at the denominator, the arc was tagged before the mutation.
    \end{enumerate}
\end{remark}

\begin{proposition}\label{prop:mutation_orbifold}
    The mutation rule \eqref{mutation_rule_orbifold} is the same as in Definition \ref{def:mutationrule} (up to some power of $q^{1/2}$).
\end{proposition}

\begin{proof}
    Since the mutation rule of Definition \ref{def:mutationrule} is recursive, we have to check that the mutation rule of \eqref{mutation_rule_orbifold} follows the same recursion. 

    Let us start with the first two steps. For the initial step, let us say we mutate at $Y$, this first mutation gives a symmetric result in $y$ and $\tilde{y}$ (for clarity, we fix the positions of $x,y,z$ in the triple): 
    \[
    (1,1,1) \longrightarrow (1,q^{1/2} + q^{-1/2},1) = (m_1,q^{-1/2}m_2,m_1),
    \]
so we indeed get the same triple as above, up to some powers of $q^{1/2}$. For the second step, there is a choice, we can either mutate at $X$ or $Z$. Since one mutation has occurred, the loops $X,Y,Z$ are now in counter-clockwise cyclic order. Assume we mutate at $X$ next, then the mutation rule \eqref{mutation_rule_orbifold} give the following triple:
\[
(1,q^{1/2} + q^{-1/2},1) \longrightarrow (q^{-3/2} + 2q^{-1/2} +2q^{1/2},q^{1/2} + q^{-1/2},1) = (q^{-3/2}m_5,q^{-1/2}m_2,m_1).
\]
We again get our mirror deformed Markov numbers, up to some powers of $q^{1/2}$. Note that if we had chosen to mutate at $Z$, we would have gotten $\tilde{m}_5$ instead, but a similar pattern (see Remark~\ref{rem:initial_choice}).

Let us now consider two consecutive mutations as in Figure \ref{fig:mutationrules} (so $z'(1) \geq x(1), y(1)$).
    
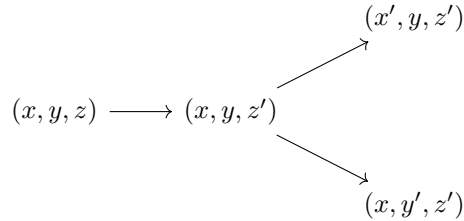
\[
    \begin{tikzcd}
        & & (x',y,z') \\
        (x,y,z) \ar[r] & (x,y,z') \ar[ru] \ar[rd]\\
        & & (x,y',z')
    \end{tikzcd}
\]
Assume that both mutation rules coincide for the first step $(x,y,z) \to (x,y,z')$, up to a power of $q^{1/2}$. 
We start with a triple $(x,y,z)$, for which we assume without loss of generality that the corresponding loops are in the configuration of the left hand side of Figure \ref{fig:mut_sphere} ($X,Y,Z$ in clockwise order), and that $(x,y,z')$ are ordered, in the sense that $x(1)\leq y(1)\leq z'(1)$. Then according to the mutation rule \eqref{mutation_rule_orbifold}, 
\[
z'=\frac{q^{1/2}X + q^{-1/2}Y}{\tilde{z}} = q^{-1/2}\frac{qX + Y}{\tilde{z}}.
\]
Hence, from Definition \ref{def:mutationrule}, with our assumptions, the mutation from $(x,y,z)$ to $(x,y,z')$ is a positive mutation (\emph{blue}). Then the coloured mutation graph is as in Figure~\ref{fig:compare_mutation_rules}.
\begin{figure}
    \centering
    \begin{tikzpicture}
        \draw[dotted] (-1.3,-4) -- (1.3,-4);
        \draw[dotted] (2.7,-4) -- (5.8,-4);
        \draw[dotted] (0.7,-2) -- (7.8,-2);
        \draw[dotted] (0.7,0) -- (5.8,0);
            \node (a) at (0,0) {$(x,y,z)$}; 
            \node (b) at (0,-2) {$(x,y,z')$}; 
            \node (c) at (-2,-4) {$(x,y',z')$}; 
            \node (d) at (2,-4) {$(x',y,z')$}; 
            \draw[very thick, blue] (a) -- (b);
            \draw[very thick, blue] (b) -- (c);
            \draw[very thick, red] (b) -- (d);
    
            \node at (7,0) {\scalebox{0.6}{\input{3_orbifold.tikz}}};
            \node at (9,-2) {\scalebox{0.6}{\input{3_orbifold_XZY.tikz}}};
            \node at (7,-4) {\scalebox{0.6}{\input{3_orbifold.tikz}}};
            \end{tikzpicture}
    \caption{Two consecutive mutations}
    \label{fig:compare_mutation_rules}
\end{figure}
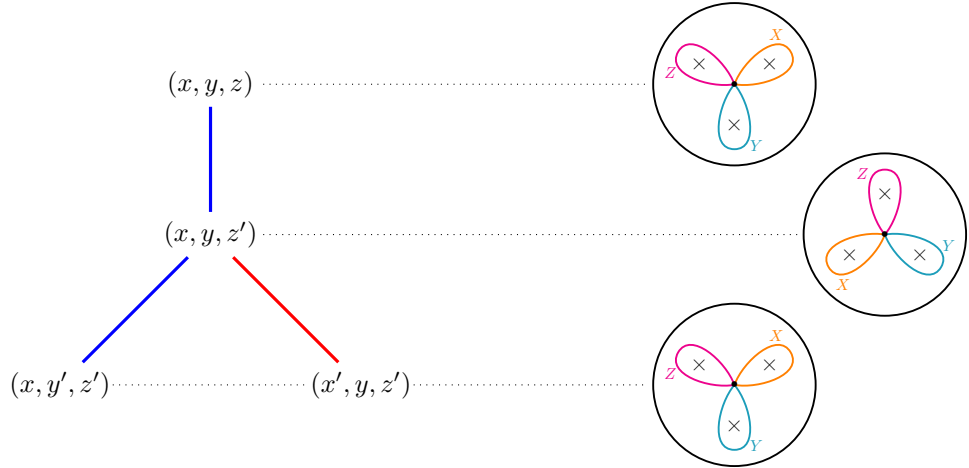

On the right are represented the configurations of the loops corresponding to $X,Y,Z$. From Figure \ref{fig:mut_sphere}, each mutation changes the configuration. We compute the new variables for the next steps using the mutation rule \eqref{mutation_rule_orbifold}:
\[
y'=\frac{q^{1/2}X + q^{-1/2}Z'}{\tilde{y}}, \quad x' = \frac{q^{1/2}Z' + q^{-1/2}Y}{\tilde{x}}.
\]
Now compare with the values obtained using the mutation rule from Definition \ref{def:positive_negative_mut}:
\[
 y'=q^{\deg z'}\frac{qX + Z'}{\tilde{y}} = q^{\deg z'+1/2}\frac{q^{1/2}X + q^{-1/2}Z'}{\tilde{y}}, \qquad x'=q^{\deg z'}\frac{q^{-1}Y + Z'}{\tilde{x}}= q^{\deg z'-1/2}\frac{q^{1/2}Z' + q^{-1/2}Y}{\tilde{x}}.
\]
We verify that the results obtained by the two different mutation rules are equal up to some powers of $q^{1/2}$, which concludes the proof.
\end{proof}

\section{Fibonacci and Pell Branches}~\label{sec:FP}

In this section we focus on Fibonacci and Pell branches of the mirror Markov tree.

\subsection{Fibonacci branch}

Fibonacci numbers are defined by $\ff_0=0$, $\ff_1=1$, $\ff_i$ for $i\geq 1$ with $\ff_{i+1}=\ff_i+\ff_{i-1}$. The odd Fibonacci numbers satisfy

\[\ff_{2i+1}=3\ff_{2i-1}-\ff_{2i-3}.\]

The left branch of the classical Markov tree contains the odd Fibonacci numbers, in triples of the form $(1,\ff_{2i-1},\ff_{2i+1})$. Let us consider the corresponding Fibonacci branch of the mirror Markov Tree, and recall Notation~\ref{not}. From now on, we denote by $f_n := m_{\ff_{2n+1}}(q)$ the $q$-deformed $n$-th odd Fibonacci number, and $F_{n} = f_n(q)f_n(q^{-1}) = q^{-n}f_n\tilde{f_n}$ the deformed squared of $\ff_{2n+1}$. The $q$-deformed odd Fibonacci numbers satisfy a similar relation.

\begin{remark}
Because of the mutation rule (see Definition \ref{def:mutationrule}), the sign of mutations on the Fibonacci branch is always positive. The mutation rule on the Fibonacci branch is
$$
f_n = q^{\deg(f_{n-1})}\frac{q+F_{n-1}}{\widetilde{f}_{n-2}}.
$$
\end{remark}

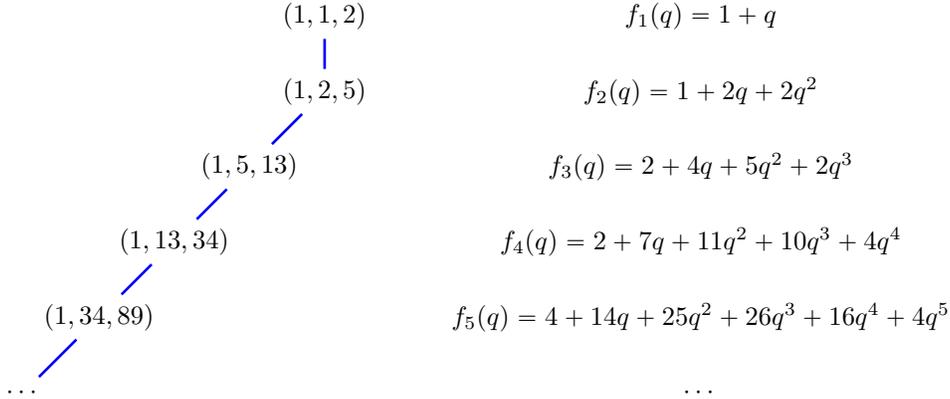
\begin{figure}[H]
    \centering
\begin{tikzpicture}
    \begin{scope}[scale=1]
    \node (A) at(0,0) {$(1,1,2)$};
    \node (B) at(0,-1) {$(1,2,5)$};
    \node (C) at(-1,-2) {$(1,5,13)$};
    \node (D) at(-2,-3) {$(1,13,34)$};
    \node (E) at(-3,-4) {$(1,34,89)$};
    \node (F) at(-4,-5) {$\cdots$};

    \draw[line width=1pt,blue] (A)-- (B);
    \draw[line width=1pt,blue] (B)-- (C);
    \draw[line width=1pt,blue] (C)-- (D);
    \draw[line width=1pt,blue] (D)-- (E);
    \draw[line width=1pt,blue] (E)-- (F);
    \end{scope}

    \begin{scope}[shift = {(5,0)},scale=1]
    \node (A) at(0,0) {$f_1(q) = 1+q$};
    \node (B) at(0,-1) {$f_2(q) = 1+2q+2q^2$};
    \node (C) at(0,-2) {$f_3(q) = 2+4q+5q^2 + 2q^3$};
    \node (D) at(0,-3) {$f_4(q) = 2 + 7q + 11q^2 + 10q^3 + 4q^4 $};
    \node (E) at(0,-4) {$f_5(q) = 4 + 14q + 25q^2 + 26q^3 + 16q^4 + 4q^5$};
    \node (F) at(0,-5) {$\cdots$};
    \end{scope}
\end{tikzpicture}
\caption{The Fibonacci branch and the $q$-deformed odd Fibonacci numbers}
\end{figure}

The classical Markov function, which is constant, is

\[\mathcal{M}(\ff_{2i+1},\ff_{2i-1},1)=\frac{\ff^2_{2i+1}+\ff_{2i-1}^2+1}{\ff_{2i+1}\ff_{2i-1}} = 3.\]

\begin{lemma}
    For $n\geq 0$, we have
    \[
    \deg(f_n)=n.
    \]
\end{lemma}

\begin{proof}
As $F_0=1$ and $F_1=q^{-1} + 2 +q$, we have $\deg(F_0)=0$ and $\deg(F_1)=1$. Then, the mutation rule for the $F_n$ is as follows, for $n\geq 1$,
\[
F_{n+1}=\frac{1 + (q+q^{-1})F_n +F_n^2}{F_{n-1}}.
\]
From here it is clear that $\deg(F_n)=n$, for all $n\geq 0$.

Whereas, for the $(f_n)$, we also have $\deg(f_0) =0$ and $\deg(f_1) =1$. Hence, with the mutation rule, for $n\geq 2$,
\[
\deg(f_n) = \deg(f_{n-1}) + n-1 - \deg(f_{n-2}),
\]
from which we deduce the result recursively.
\end{proof}

\begin{theorem}\label{proposition: q_markov function for Fibonacci}
The $q$-Markov function
\[
\mathcal{M}^f_q(1,f_n,f_{n-1})=\frac{q^{n}F_0+q^{n+1}F_{n}+q^{n-1}F_{n-1}}{f_{n}f_{n-1}}
\]
is constant equal to $[3]_q = 1+q+q^2$, for $n\geq 1$.
\end{theorem}

\begin{proof}
    By induction on $n \geq 1$. \\
\noindent For $n= 1$, we check that
\begin{align*}
qF_0 + q^2F_1 + F_0 &= q + q^2(1+q)(1+q^{-1}) + 1\\
&= [3]_q(1+q) = [3]_qf_1f_0.\\
\end{align*}
\noindent Now take $n \geq 2$ such that $\mathcal{M}^f_q(1,f_{n-1},f_{n-2}) = [3]_q$.

Then, from the mutation rule:
\[
f_n= q^{n-1}\frac{qF_0 + F_{n-1}}{\tilde{f}_{n-2}}.
\]

Using this mutation formula and the induction hypothesis, we get
\begin{align*}
    q^{n}F_0 + q^{n+1}F_{n} + q^{n-1}F_{n-1} &= q^n + q^{n+1}\frac{1 + (q+q^{-1})F_{n-1} + F_{n-1}^2}{F_{n-2}} + q^{n-1}F_{n-1}\\
    &=\frac{q^nF_{n-2} +q^{n+2}F_{n-1} +q^{n+1} }{F_{n-2}} + \frac{F_{n-1} }{F_{n-2}} \left( q^n + q^{n+1}F_{n-1} + q^{n-1}F_{n-2}\right)\\
    &=q^2\frac{[3]_qf_{n-1}f_{n-2}}{F_{n-2}} + q\frac{F_{n-1} }{F_{n-2}}\left( [3]_qf_{n-1}f_{n-2}\right)\\
    &=[3]_qf_{n-1}\frac{q^2 +qF_{n-1}}{q^{-n+2}\tilde{f}_{n-2}} = [3]_qf_{n-1}f_{n}.
\end{align*}
\end{proof}

\begin{remark}
    Similarly, for all $n\geq 1$, we have
    \[
    [3]_q\tilde{f}_{n}\tilde{f}_{n-1} = q^{n+1}F_0 +q^nF_n + q^{n+2}F_{n-1}.
    \]
\end{remark}

\begin{proposition} 
The $q$-deformed odd Fibonacci numbers satisfy, for $n\geq 1$,
\begin{align*}
    f_{n+1}& =[3]_q\tilde{f}_{n} -q^3f_{n-1},\\
    q\tilde{f}_{n+1} &= [3]_qf_n - \tilde{f}_{n-1}.
\end{align*}
\end{proposition}

\begin{proof}
From the mutation rule, and the previous remark,
\begin{align*}
        f_{n+1} &= q^{n} \frac{qF_0 +F_{n}}{\tilde{f}_{n-1}} = \frac{q^{n+1}F_0}{\tilde{f}_{n-1}} + \frac{[3]_q\tilde{f}_{n}\tilde{f}_{n-1} - q^{n+1}F_0 - q^{n+2}F_{n-1}}{\tilde{f}_{n-1}} \\
        &= [3]_q\tilde{f}_{n} -q^{n+2}\frac{q^{n-1}\tilde{f}_{n-1}f_{n-1}}{\tilde{f}_{n-1}} = [3]_q\tilde{f}_{n} -q^3f_{n-1}.
\end{align*}
The second assertion is deduced from the first by taking $\sim$, for all $n\geq 1$:
\begin{align*}
    \tilde{f}_{n+1}(q) &= q^{n+1}f_{n+1}(q^{-1})  = q^{n+1}([3]_{q^{-1}}\tilde{f}_n(q^{-1}) -q^{-3}f_{n-1}(q^{-1})) \\
    &= q^{n+1}(q^{-n-2}[3]_qf_{n}(q) -q^{-n-2}\tilde{f}_{n-1}(q))\\
    &= q^{-1}([3]_qf_n - \tilde{f}_{n-1}).
\end{align*}
\end{proof}

\begin{proposition}\label{prop:positivity Fibonacci}
The coefficients in the $q$-deformed odd Fibonacci numbers are positive.
\end{proposition}

\begin{proof} As in Section~\ref{sec:spec int}, we use continued fractions to write fractions of $q$-deformed Fibonacci numbers. We have
\[
\begin{array}{ccl}
    \frac{f_1}{\tilde{f}_0}=1+q = [1+q] & \text{and} & \frac{f_2}{\tilde{f}_1}=\frac{1+2q+2q^2}{1+q} = 1+q + \frac{q^2}{1+q} = [1+q,q^{-2}(1+q)], \\
     \frac{\tilde{f}_1}{f_0}=1+q = [1+q] & \text{and} & \frac{q\tilde{f}_2}{f_1}=\frac{2q+2q^2+q^3}{1+q} = q+q^2 + \frac{q}{1+q} = [q+q^2,q^{-1}(1+q)].
\end{array}
\]

In general, for $n\geq 2$,
\begin{align*}
    \frac{f_{n+1}}{\tilde{f}_n} & = \frac{[3]_q\tilde{f}_{n} -q^3f_{n-1}}{\tilde{f}_n} = 1+q + q^2\frac{\tilde{f}_{n} -qf_{n-1}}{\tilde{f}_n} = 1+q + \frac{1}{\frac{q^{-2}\tilde{f}_n}{\tilde{f}_{n} -qf_{n-1}}}\\
    & = [1+q , \frac{q^{-2}\tilde{f}_n}{\tilde{f}_{n} -qf_{n-1}}] = [1+q , q^{-2} +\frac{q^{-1}f_{n-1}}{\tilde{f}_{n} -qf_{n-1}}] = [1+q , q^{-2} ,\frac{\tilde{f}_{n} -qf_{n-1}}{q^{-1}f_{n-1}}] \\  
    &= [1+q , q^{-2} ,\frac{q\tilde{f}_{n}}{f_{n-1}}-q^2],
\end{align*}
and 
\begin{align*}
    \frac{q\tilde{f}_{n+1}}{f_n} & = \frac{[3]_qf_n - \tilde{f}_{n-1}}{f_n} = q+q^2+ \frac{f_n - \tilde{f}_{n-1}}{f_n} = q+q^2+ \frac{1}{\frac{f_n}{f_n - \tilde{f}_{n-1}}} \\
    & = [q+q^2, \frac{f_n}{f_n - \tilde{f}_{n-1}}] = [q+q^2, 1 + \frac{\tilde{f}_{n-1}}{f_n - \tilde{f}_{n-1}}]  = [q+q^2, 1 ,  \frac{f_n - \tilde{f}_{n-1}}{\tilde{f}_{n-1}}] \\
    & = [q+q^2, 1 ,  \frac{f_n }{\tilde{f}_{n-1}}-1].
\end{align*}
Hence, for all $n\geq 1$,
\[
\frac{f_{n+1}}{\tilde{f}_n} = \left\lbrace \begin{array}{cc}
    [1+q,q^{-2}, q, 1, q, q^{-2}, q, 1, q, \ldots, q^{-2}, q, q^{-1}(1+q) ], & \text{if $n$ is even}, \\
    {[}1+q,q^{-2}, q, 1, q, q^{-2}, q, 1, q, \ldots, 1, q, q^{-2}(1+q) ], & \text{if $n$ is odd}.
\end{array}\right.
\]

This implies the positivity of the coefficients of all $f_n$, for $n\geq 0$.
\end{proof}

\subsection{Pell branch}

There is also the $q$-deformed Pell branch of our tree. The usual Pell recurrence is as follows.

\[\mathcal{P}_{n+1}=2\mathcal{P}_n+\mathcal{P}_{n-1}\] with $\mathcal{P}_0=0$ and $\mathcal{P}_1=1$. The odd Pell numbers are in the left branch of the right part of the Markov tree. Similar to the previous subsection, we now consider the corresponding Pell branch of the mirror Markov Tree. For $n\geq 0$, we will denote by $p_n$ the $q$-deformed $(2n+1)$-th Pell numbers and by $P_n = p_n(q)p_n(q^{-1})$ the deformed squared. Note that, for $n\geq 0$,
\[
p_n(q)= m_{\mathcal{P}_{2n+1}}(q), 
\]

\noindent and the mutation rule on this branch is 
$$
p_n = q^{\deg(p_{n-1})}\frac{q^{-1}M_2 + P_{n-1}}{\tilde{p}_{n-2}}.
$$

\begin{figure}[H]
    \centering
\begin{tikzpicture}
    \begin{scope}[scale=1]
    \node (Z) at(0,1) {$(1,1,2)$};
    \node (A) at(0,0) {$(1,2,5)$};
    \node (B) at(1,-1) {$(2,5,29)$};
    \node (C) at(0,-2) {$(2,29,169)$};
    \node (D) at(-1,-3) {$(2,169,985)$};
    \node (E) at(-2,-4) {$(2,985,5741)$};
    \node (F) at(-3,-5) {$\cdots$};

    \draw[line width=1pt,red] (A)-- (B);
    \draw[line width=1pt,red] (B)-- (C);
    \draw[line width=1pt,red] (C)-- (D);
    \draw[line width=1pt,red] (D)-- (E);
    \draw[line width=1pt,red] (E)-- (F);
    \draw[line width=1pt,blue] (Z)-- (A);
    \end{scope}

    \begin{scope}[shift = {(6,0)},scale=1]
    \node (Z) at(0,1) {$p_0(q) = 1$};
    \node (A) at(0,0) {$p_1(q) = 1+2q+2q^2$};
    \node (B) at(0,-1) {$p_2(q) = 3 + 8q + 10q^2 + 6q^3 + 2q^4$};
    \node (C) at(0,-2) {$p_3(q) = 3 + 14q + 34q^2 + 48q^3 + 42q^4 + 22q^5 + 6q^6$};
    \node (D) at(0,-3) {$p_4(q) = 9 + 48q + 130q^2 + 218q^3 + 246q^4 + 192q^5 + 102q^6 + 34q^7 + 6q^8$};
    \node (E) at(0,-4) {$\cdots$};
    \end{scope}
\end{tikzpicture}
\caption{The Pell branch and the $q$-deformed odd Pell numbers}
\end{figure}
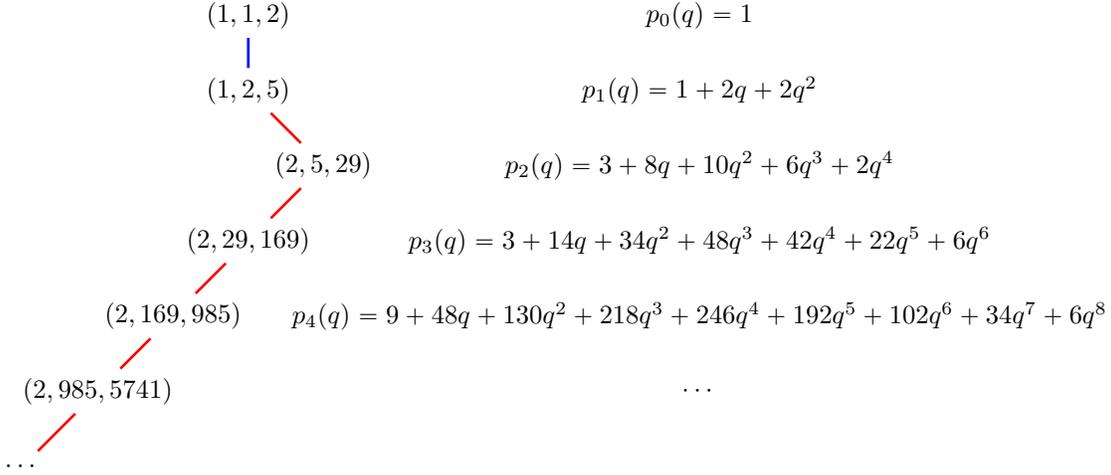

\begin{lemma}
    For $n\geq 0$, we have
    \[
    \deg(p_n)=2n.
    \]
\end{lemma}

\begin{proof}
As $P_0=1$ and $P_1=2q^{-2} + 6q^{-1} + 9 + 6q + 2q^2
$, we have $\deg(P_0)=0$ and $\deg(P_1)=2$. Then, the mutation rule for the $P_n$ is as follows, for $n\geq 1$,
\[
P_{n+1}=\frac{M_2^2 + (q+q^{-1})M_2P_n +P_n^2}{P_{n-1}}.
\]
From here it is clear that $\deg(P_n)=2n$, for all $n\geq 0$.

Whereas, for the $(p_n)$, we also have $\deg(p_0) =0$ and $\deg(p_1) =2$. Hence, with the mutation rule, for $n\geq 2$,
\[
\deg(p_n) = \deg(p_{n-1}) + 2n-2 - \deg(p_{n-2}),
\]
from which we deduce the result recursively.
\end{proof}

\begin{theorem}\label{proposition: q_markov function for Pell}
The $q$-Markov function
\[
\mathcal{M}^p_q(2,p_n,p_{n-1})=\frac{q^{2n+1}M_2+q^{2n}P_{n}+q^{2n+2}P_{n-1}}{2p_{n}p_{n-1}}
\]
is constant equal to $[3]_q = 1+q+q^2$, for $n\geq 1$.
\end{theorem}

\begin{proof}
    We prove it by induction on $n$. \\
\noindent For $n=1$, we check that
\begin{align*}
    q^3M_2 +q^2P_1 + q^4P_0 &= q^3(q^{-1}+2+q) +q^2(1 + 2q + 2q^2)(1+2q^{-1} +2q^{-2}) + q^4\\
    &= 2 + 6q + 10q^2 + 8q^3 + 4q^4= 2(1+q+q^2)(1+2q+2q^2)\\ &= [3]_q(2p_1p_0)
\end{align*}
\noindent Now take $n \geq 2$ such that $\mathcal{M}^p_q(2,p_{n-1},p_{n-2}) = [3]_q$.

Then, from the result of the previous lemma and the mutation rule we get:
\[
    p_{n} = q^{2n-2} \frac{q^{-1}M_2 +P_{n-1}}{\tilde{p}_{n-2}}.
    \]
    
Using this mutation formula and the induction hypothesis, we get
\begin{align*}
    q^{2n+1}M_2+q^{2n}P_{n}+q^{2n+2}P_{n-1} =& \,q^{2n+1}M_2+ q^{2n}\frac{M_2^2 + (q+q^{-1})M_2P_{n-1} + P_{n-1}^2}{P_{n-2}} +q^{2n+2}P_{n-1} \\
    =&\,  q\frac{M_2}{P_{n-2}}\left( q^{2n}P_{n-2} + q^{2n-1}M_2 +q^{2n-2}P_{n-1}\right) \\
    & + q^2\frac{P_{n-1}}{P_{n-2}}\left( q^{2n}P_{n-2} + q^{2n-1}M_2 +q^{2n-2}P_{n-1} \right)\\
    =&\, q^2\frac{q^{-1}M_2 + P_{n-1}}{q^{-2n+4}p_{n-2}\tilde{p}_{n-2}}[3]_q2p_{n-1}p_{n-2} = [3]_q2p_{n}p_{n-1}.
\end{align*}
\end{proof}

\begin{remark}
    Similarly, for all $n\geq 1$, we have
    \[
    [3]_q(2\tilde{p}_{n}\tilde{p}_{n-1}) = q^{2n-1}M_2+q^{2n}P_{n}+q^{2n-2}P_{n-1}.
    \]
\end{remark}

\begin{proposition} 
The $q$-deformed odd Pell numbers satisfy, for $n\geq 1$,
\begin{align*}
    p_{n+1} &=2[3]_q\tilde{p}_{n} -p_{n-1}\\
    \tilde{p}_{n+1} &=2[3]_qp_n - q^4\tilde{p}_{n-1}.
\end{align*}
\end{proposition}

\begin{proof}
    From the previous remark,
\begin{align*}
    p_{n+1} & = q^{2n}\frac{q^{-1}M_2 +P_n}{\tilde{p}_{n-1}} = \frac{q^{2n-1}M_2 }{\tilde{p}_{n-1}} + \frac{2[3]_q\tilde{p}_n\tilde{p}_{n-1}-q^{2n-1}M_2 -q^{2n-2}P_{n-1}}{\tilde{p}_{n-1}} \\
    &= 2[3]_q\tilde{p}_{n} - p_{n-1}.
\end{align*}
The second assertion is deduced from the first by taking $\sim$, knowing that the difference of degree between $p_{n+1}$ and $p_{n-1}$ is 4.
\end{proof}

\begin{proposition}
    The coefficients of the $q$-deformed odd Pell numbers are positive.
\end{proposition}

\begin{proof}
    The proof is similar to that of Proposition~\ref{prop:positivity Fibonacci} for the Fibonacci branch. We have
\[
\frac{p_1}{\tilde{p}_0}=1+2q+2q^2 = [p_2] \quad \text{and} \quad  \frac{\tilde{p}_1}{f_0}=2+2q+q^2  = [\tilde{p}_2] .
\]  
\begin{align*}
\frac{p_2}{\tilde{p}_1}& =\frac{3 + 8q + 10q^2 + 6q^3 + 2q^4}{2+2q+q^2} =  1+2q+2q^2 + \frac{1+2q+q^2}{2+2q+q^2}\\
& = [1+2q+2q^2, \frac{2+2q+q^2}{1+2q+q^2}] = [1+2q+2q^2, 1+\frac{1}{1+2q+q^2}] = [p_2, 1, (1+q)^2],\\
\frac{\tilde{p}_2}{p_1}& =\frac{2 + 6q + 10q^2 + 8q^3 + 3q^4}{1+2q+2q^2} =  2+2q+q^2 +\frac{q^2+2q^3+q^4}{1+2q+2q^2}\\
& = [2+2q+q^2, \frac{1+2q+2q^2}{q^2(1+2q+q^2)}] = [2+2q+q^2, q^{-2} + \frac{1}{1+2q+q^2}] = [\tilde{p}_2, q^{-2}, (1+q)^2].
\end{align*}
In general, for $n\geq 2$,
\begin{align*}
\frac{p_{n+1}}{\tilde{p}_n}& = \frac{2[3]_q\tilde{p}_{n} -p_{n-1}}{\tilde{p}_n} =  1+2q+2q^2 + \frac{\tilde{p}_{n} -p_{n-1}}{\tilde{p}_n} = [p_2, \frac{\tilde{p}_n}{\tilde{p}_{n} -p_{n-1}}] \\
& =[p_2, 1 + \frac{p_{n-1}}{\tilde{p}_{n} -p_{n-1}}]  = [p_2, 1, \frac{\tilde{p}_{n}}{p_{n-1}}-1],\\
\frac{\tilde{p}_{n+1}}{p_n}& = \frac{2[3]_qp_{n} -q^4 \tilde{p}_{n-1}}{p_n} = 2+2q+q^2 + \frac{q^2p_{n} -q^4\tilde{p}_{n-1}} {p_n} = [\tilde{p}_2, \frac{p_n}{q^2p_{n} -q^4\tilde{p}_{n-1}}] \\ 
& = [\tilde{p}_2, q^{-2} +\frac{\tilde{p}_{n-1}}{p_{n} -q^2\tilde{p}_{n-1}}]  = [\tilde{p}_2, q^{-2},\frac{p_{n} }{\tilde{p}_{n-1}}-q^2] .
\end{align*}
Hence, for all $n\geq 1$, $\frac{p_{n+1}}{\tilde{p}_n}$ can be written as a continued fraction of length $2n+1$, which is
\[
\frac{p_{n+1}}{\tilde{p}_n} = 
     [p_2, 1, (1+q)^2, q^{-2}, (1+q)^2, \ldots,(1+q)^2] .
\]
This implies the positivity of the coefficients of all $p_n$, for $n\geq 0$.
\end{proof}

\begin{remark} Having proved the positivity for both Fibonacci and Pell branches,  we accomplish strong hints to the positivity holds for all branches, see Conjecture~\ref{conj:positivity}.
\end{remark}

\section{Open questions}~\label{sec:open}

For a fixed Markov number $\ell\geq 5$, there are two branches on the classical Markov tree of the form $(\ell,b,c)$, with $\ell$ fixed; one blue branch and one red branch. We conjecture that the phenomenon observed for the $q$-Markov functions on the Fibonacci and Pell branches generalizes to these branches. 

For $m_\ell(q)$ the mirror deformation of $\ell$, let $(p_n)_{n\geq 0}$ be the sequence of polynomials appearing on one of the branches of the mirror Markov tree which fixes $m_\ell$. 
By convention, we have $p_0=m_\ell$. As before, we use the uppercase notations for the symmetrized polynomials: $M=m_\ell(q)m_\ell(q^{-1})$, $P_n=p_n(q)p_n(q^{-1})$, $n\geq 0$. 

\begin{conjecture}
There exists a (possibly different) $q$-deformation $a_\ell(q) \in \ZZ_{\geq 0}[q]$ of $\ell$ such that, on the branch $(m_\ell,p_{n-1},p_n)$ with $\ell$ fixed, if we define the following $q$-Markov function
\[
\mathcal{M}_q^{\ell}(m_\ell,p_{n-1},p_n) = \dfrac{q^{d}M + q^{d+\varepsilon}P_n + q^{d-\varepsilon}P_{n-1}}{a_\ell p_np_{n-1}}
\]
where $d=\deg(p_{n+1})+\varepsilon$, and $\varepsilon=1$ if the branch is red and $\varepsilon=-1$ if the branch is blue, then $\mathcal{M}_q^{\ell}(m_\ell,p_{n-1},p_n)$ is constant equal to $[3]_q$.
\end{conjecture}

We obtain a conjectural formula for $a_\ell(q)$ based on computer calculations. But we refrain to state it here as this computation favours some choices at the beginning.

\begin{conjecture}~\label{conj:positivity}
    Experimentally, the coefficients of the mirror deformed Markov numbers, $m_n$, have always been positive. We conjecture that these coefficients are all positive. 
\end{conjecture}

\begin{conjecture}
We conjecture a \emph{$q$-deformed uniqueness} for the mirror deformation, i.e. two mirror deformed triple can not have the same maximum-degree polynomial.
\end{conjecture}
 
\bibliographystyle{alpha}
\bibliography{references}

\end{document}

%% file: mutation_orbifold.tikz
\begin{tikzpicture}
    \node at (-6,0) {\input{3_orbifold.tikz}};
    \node at (-3,0) {$\xrightarrow{\mu_X}$};
        \node at (0,0) {\begin{tikzpicture}[scale=1.2]
    \draw[very thick, teal] (0,0) to[out=-130,in=180,] (0,-1.2) to[out=0,in=-50] (0,0);
    \draw[very thick, orange] (0,0) to[out=-160,in=180] (0,-1.3) to[out=0,in=-70] (0.45,0)  to[out=110,in=180] (50:1) to[out=0,in=90] (1.1,0) to[out=-90,in=0] (0,-1.4) to[out=180,in=-90] (220:0.9) to[out=90,in=190] (0,0);
    \draw[very thick, magenta] (0,0) to[out=110,in=60] (150:1.2) to[out=240,in=190] (0,0);
    \draw node at (0,0) {$\bullet$};
    \draw node at (0,-0.75) {$\bigtimes$};
    \draw node at (30:0.75) {$\bigtimes$};
    \draw node at (150:0.75) {$\bigtimes$};
    \draw[very thick] (0,0) circle (1.5);
    \draw[orange] node at (60:1.1) {$X$};
    \draw[teal] node at (-45:0.5) {$Y$};
     \draw[magenta] node at (170:1.2) {$Z$};
\end{tikzpicture}};
\node at (3,0) {$\simeq$} ;
\node at (6,0) {\input{3_orbifold_XZY.tikz}};
    \end{tikzpicture}

%% file: 3_orbifold.tikz
\begin{tikzpicture}[scale=1.2]
    \draw[very thick, teal] (0,0) to[out=-130,in=180,] (0,-1.2) to[out=0,in=-50] (0,0);
    \draw[very thick, orange] (0,0) to[out=-10,in=300] (30:1.2) to[out=120,in=70] (0,0);
    \draw[very thick, magenta] (0,0) to[out=110,in=60] (150:1.2) to[out=240,in=190] (0,0);
    \draw node at (0,0) {$\bullet$};
    \draw node at (0,-0.75) {$\bigtimes$};
    \draw node at (30:0.75) {$\bigtimes$};
    \draw node at (150:0.75) {$\bigtimes$};
    \draw[very thick] (0,0) circle (1.5);
    \draw[orange] node at (50:1.2) {$X$};
    \draw[teal] node at (-70:1.2) {$Y$};
     \draw[magenta] node at (170:1.2) {$Z$};
\end{tikzpicture}

%% file: 3_orbifold_XZY.tikz
\begin{tikzpicture}[scale=1.2]
    \draw[very thick, teal] (0,0) to[out=-70,in=240,] (-30:1.2) to[out=60,in=10] (0,0);
    \draw[very thick, orange] (0,0) to[out=170,in=120] (210:1.2) to[out=-60,in=250] (0,0);
    \draw[very thick, magenta] (0,0) to[out=50,in=0] (90:1.2) to[out=180,in=130] (0,0);
    \draw node at (0,0) {$\bullet$};
    \draw node at (-30:0.75) {$\bigtimes$};
    \draw node at (210:0.75) {$\bigtimes$};
    \draw node at (90:0.75) {$\bigtimes$};
    \draw[very thick] (0,0) circle (1.5);
    \draw[orange] node at (230:1.2) {$X$};
    \draw[teal] node at (-10:1.2) {$Y$};
     \draw[magenta] node at (110:1.2) {$Z$};
\end{tikzpicture}

%% file: mutation_pair_arcs.tikz
\begin{tikzpicture}
    \node at (-6,0) {\input{2_orbifolds_2_punctures.tikz}};
    \node at (-3,0) {$\xrightarrow{\mu_X}$};
        \node at (0,0) {\begin{tikzpicture}[scale=1.2]
    \draw[very thick, teal] (0,0) to[out=-130,in=180,] (0,-1.2) to[out=0,in=-50] (0,0);
    \draw[very thick, orange,postaction={decorate,decoration={markings,
      mark=at position 0.8 with {\node[sloped, allow upside down] {$\bowtie$};}}}] (0,0) to[out=-160,in=180] (0,-1.3) to[out=0,in=-70] (0.5,0)  to[out=110,in=180] (30:1);
    \draw[very thick, orange] (30:1) to[out=0,in=90] (1.1,0) to[out=-90,in=0] (0,-1.4) to[out=180,in=-90] (220:0.9) to[out=90,in=190] (0,0);
    \draw[very thick, magenta] (0,0) to[out=110,in=60] (150:1.2) to[out=240,in=190] (0,0);
    \draw node at (0,0) {$\bullet$};
    \draw node at (0,-0.75) {$\bigtimes$};
    \draw node at (30:1) {$\bullet$};
    \draw node at (150:0.75) {$\bigtimes$};
    \draw[very thick] (0,0) circle (1.5);
    \draw[orange] node at (45:0.8) {$\tilde{x}$};
    \draw[orange] node at (0:1.3) {$x$};
    \draw[teal] node at (-45:0.5) {$Y$};
     \draw[magenta] node at (170:1.2) {$Z$};
\end{tikzpicture}};
\end{tikzpicture}

%% file: 2_orbifolds_2_punctures.tikz
\begin{tikzpicture}[scale=1.2]
    \draw[very thick, teal] (0,0) to[out=-130,in=180,] (0,-1.2) to[out=0,in=-50] (0,0);
    \draw[very thick, orange,postaction={decorate,decoration={markings,
      mark=at position 0.8 with {\node[sloped, allow upside down] {$\bowtie$};}}}] (0,0) to[out=-10,in=300] (30:1); \draw[very thick, orange] (30:1) to[out=120,in=70] (0,0);
    \draw[very thick, magenta] (0,0) to[out=110,in=60] (150:1.2) to[out=240,in=190] (0,0);
    \draw node at (0,0) {$\bullet$};
    \draw node at (0,-0.75) {$\bigtimes$};
    \draw node at (30:1) {$\bullet$};
    \draw node at (150:0.75) {$\bigtimes$};
    \draw[very thick] (0,0) circle (1.5);
    \draw[orange] node at (60:0.8) {$x$};
    \draw[orange] node at (0:0.8) {$\tilde{x}$};
    \draw[teal] node at (-70:1.2) {$Y$};
     \draw[magenta] node at (170:1.2) {$Z$};
\end{tikzpicture}

%% file: references.bib
@book{Aigner2013MarkovBook,
    AUTHOR = {Aigner, Martin},
     TITLE = {Markov's theorem and 100 years of the uniqueness conjecture},
      NOTE = {A mathematical journey from irrational numbers to perfect
              matchings},
 PUBLISHER = {Springer, Cham},
      YEAR = {2013},
     PAGES = {x+257},
      ISBN = {978-3-319-00887-5; 978-3-319-00888-2},
   MRCLASS = {11-03 (11J06 20E05 20H10 68R15)},
  MRNUMBER = {3098784},
MRREVIEWER = {Thomas\ A.\ Schmidt},
       DOI = {10.1007/978-3-319-00888-2},
       URL = {https://doi.org/10.1007/978-3-319-00888-2},
}

@article{GKW24,
AUTHOR = {Greenberg, Zachary and Kaufman, Dani and Wienhard, Anna},
     TITLE = {{$\rm SL_2$}-like properties of matrices over noncommutative
              rings and generalizations of {M}arkov numbers},
   JOURNAL = {Adv. Math.},
  FJOURNAL = {Advances in Mathematics},
    VOLUME = {482},
      YEAR = {2025},
     PAGES = {Paper No. 110556, 62},
      ISSN = {0001-8708,1090-2082},
   MRCLASS = {20G42 (11D25 13F60 15B33 16W10)},
  MRNUMBER = {4968025},
       DOI = {10.1016/j.aim.2025.110556},
       URL = {https://doi.org/10.1016/j.aim.2025.110556},
}

@misc{G22,
      title={Positive integer solutions to $(x+y)^2+(y+z)^2+(z+x)^2=12xyz$}, 
      author={Gyoda, Y.},
      year={2022},
      eprint={2109.09639v4},
      archivePrefix={arXiv},
      primaryClass={math.CO},
      url={https://arxiv.org/pdf/2109.09639}, 
    note={\url{https://arxiv.org/pdf/2109.09639}},
}

@article {GM23,
    AUTHOR = {Gyoda, Yasuaki and Matsushita, Kodai},
     TITLE = {Generalization of {M}arkov {D}iophantine equation via
              generalized cluster algebra},
   JOURNAL = {Electron. J. Combin.},
  FJOURNAL = {Electronic Journal of Combinatorics},
    VOLUME = {30},
      YEAR = {2023},
    NUMBER = {4},
     PAGES = {Paper No. 4.10, 20},
      ISSN = {1077-8926},
   MRCLASS = {11D25 (13F60)},
  MRNUMBER = {4657283},
MRREVIEWER = {Arthur\ Baragar},
       DOI = {10.37236/11420},
       URL = {https://doi.org/10.37236/11420},
}

@misc{GMS,
      title={SL(2,Z)-matrixizations of generalized Markov numbers}, 
      author={Yasuaki Gyoda and Shuhei Maruyama and Yusuke Sato},
      year={2024},
      eprint={2407.08203},
      archivePrefix={arXiv},
      primaryClass={math.NT},
      url={https://arxiv.org/abs/2407.08203}, 
    note={\url{https://arxiv.org/pdf/2407.08203}},
}

@article {M79,
    AUTHOR = {Markoff, A.},
     TITLE = {Sur les formes quadratiques binaires ind\'efinies},
   JOURNAL = {Math. Ann.},
  FJOURNAL = {Mathematische Annalen},
    VOLUME = {15},
      YEAR = {1879},
    NUMBER = {3-4},
     PAGES = {381--406},
      ISSN = {0025-5831,1432-1807},
   MRCLASS = {11E16 (11A55)},
  MRNUMBER = {4788527},
}

@misc{gyoda_uniqueness_2024,
	title = {Uniqueness theorem of generalized {Markov} numbers that are prime powers},
	url = {http://arxiv.org/abs/2312.07329},
	doi = {10.48550/arXiv.2312.07329},
	urldate = {2025-05-01},
	publisher = {arXiv},
	author = {Gyoda, Yasuaki and Maruyama, Shuhei},
	month = jul,
	year = {2024},
    note={\url{https://arxiv.org/abs/2312.07329}}
}

@misc{M25,
      title={Super Markov Numbers and Signed Double Dimer Covers}, 
      author={Musiker, G.},
      year={2025},
      eprint={2503.21872},
      archivePrefix={arXiv},
      primaryClass={math.CO},
      url={https://www.arxiv.org/pdf/2503.21872}, 
    note={\url{https://arxiv.org/pdf/2503.21872}},
}

@article{Leclere2021Q-DeformationsNumbers,
    title = {{q-Deformations in the modular group and of the real quadratic irrational numbers}},
    year = {2021},
    journal = {Advances in Applied Mathematics},
    author = {Leclere, Ludivine and Morier-Genoud, Sophie},
    volume = {130},
    doi = {10.1016/j.aam.2021.102223},
    issn = {10902074}
}

@article{kantarci_oguz_oriented_2025,
	title = {Oriented posets, rank matrices and q-deformed {Markov} numbers},
	volume = {348},
	issn = {0012365X},
	url = {https://linkinghub.elsevier.com/retrieve/pii/S0012365X2400387X},
	doi = {10.1016/j.disc.2024.114256},
	language = {en},
	number = {2},
	urldate = {2025-05-01},
	journal = {Discrete Mathematics},
	author = {Kantarcı Oğuz, Ezgi},
	month = feb,
	year = {2025},
	pages = {114256},
}

@misc{Kogiso2020,
      title={$q$-Deformations and $t$-deformations of Markov triples}, 
      author={Takeyoshi Kogiso},
      year={2020},
      eprint={2008.12913},
      archivePrefix={arXiv},
      primaryClass={math.NT},
      url={https://arxiv.org/abs/2008.12913}, 
    note={\url{https://arxiv.org/pdf/2008.12913}}
}

@article {Huang2023,
    AUTHOR = {Huang, Yi and Penner, Robert C. and Zeitlin, Anton M.},
     TITLE = {Super {M}c{S}hane identity},
   JOURNAL = {J. Differential Geom.},
  FJOURNAL = {Journal of Differential Geometry},
    VOLUME = {125},
      YEAR = {2023},
    NUMBER = {3},
     PAGES = {509--551},
      ISSN = {0022-040X,1945-743X},
   MRCLASS = {57M05 (37E30)},
  MRNUMBER = {4674074},
MRREVIEWER = {Sadayoshi\ Kojima},
       DOI = {10.4310/jdg/1701804150},
       URL = {https://doi.org/10.4310/jdg/1701804150},
}

@article {Musiker2023,
    AUTHOR = {Musiker, Gregg and Ovenhouse, Nicholas and Zhang, Sylvester
              W.},
     TITLE = {Matrix formulae for decorated super {T}eichm\"uller spaces},
   JOURNAL = {J. Geom. Phys.},
  FJOURNAL = {Journal of Geometry and Physics},
    VOLUME = {189},
      YEAR = {2023},
     PAGES = {Paper No. 104828, 28},
      ISSN = {0393-0440,1879-1662},
   MRCLASS = {30F60 (13F60 58A50)},
  MRNUMBER = {4575094},
       DOI = {10.1016/j.geomphys.2023.104828},
       URL = {https://doi.org/10.1016/j.geomphys.2023.104828},
}

@article {P87,
    AUTHOR = {Penner, R. C.},
     TITLE = {The decorated {T}eichm\"uller space of punctured surfaces},
   JOURNAL = {Comm. Math. Phys.},
  FJOURNAL = {Communications in Mathematical Physics},
    VOLUME = {113},
      YEAR = {1987},
    NUMBER = {2},
     PAGES = {299--339},
      ISSN = {0010-3616,1432-0916},
   MRCLASS = {32G15 (14H15 53A35)},
  MRNUMBER = {919235},
MRREVIEWER = {C.\ Earle},
       URL = {http://projecteuclid.org/euclid.cmp/1104160216},
}

@article {FST08,
    AUTHOR = {Fomin, Sergey and Shapiro, Michael and Thurston, Dylan},
     TITLE = {Cluster algebras and triangulated surfaces. {I}. {C}luster
              complexes},
   JOURNAL = {Acta Math.},
  FJOURNAL = {Acta Mathematica},
    VOLUME = {201},
      YEAR = {2008},
    NUMBER = {1},
     PAGES = {83--146},
      ISSN = {0001-5962,1871-2509},
   MRCLASS = {57Q15 (13F60 32G15 52B70)},
  MRNUMBER = {2448067},
MRREVIEWER = {Christof\ Gei\ss},
       DOI = {10.1007/s11511-008-0030-7},
       URL = {https://doi.org/10.1007/s11511-008-0030-7},
}

@article {FT18,
    AUTHOR = {Fomin, Sergey and Thurston, Dylan},
     TITLE = {Cluster algebras and triangulated surfaces {P}art {II}:
              {L}ambda lengths},
   JOURNAL = {Mem. Amer. Math. Soc.},
  FJOURNAL = {Memoirs of the American Mathematical Society},
    VOLUME = {255},
      YEAR = {2018},
    NUMBER = {1223},
     PAGES = {v+97},
      ISSN = {0065-9266,1947-6221},
      ISBN = {978-1-4704-2967-6; 978-1-4704-4823-3},
   MRCLASS = {13F60 (30F60 57M50)},
  MRNUMBER = {3852257},
MRREVIEWER = {Christof\ Gei\ss},
       DOI = {10.1090/memo/1223},
       URL = {https://doi.org/10.1090/memo/1223},
}

@article {PZ19,
    AUTHOR = {Penner, R. C. and Zeitlin, Anton M.},
     TITLE = {Decorated super-{T}eichm\"uller space},
   JOURNAL = {J. Differential Geom.},
  FJOURNAL = {Journal of Differential Geometry},
    VOLUME = {111},
      YEAR = {2019},
    NUMBER = {3},
     PAGES = {527--566},
      ISSN = {0022-040X,1945-743X},
   MRCLASS = {32C11 (32G15 57N05 58A50)},
  MRNUMBER = {3934599},
MRREVIEWER = {Ya\c sar\ S\"ozen},
       DOI = {10.4310/jdg/1552442609},
       URL = {https://doi.org/10.4310/jdg/1552442609},
}

@article {MOZ21,
    AUTHOR = {Musiker, Gregg and Ovenhouse, Nicholas and Zhang, Sylvester
              W.},
     TITLE = {An expansion formula for decorated super-{T}eichm\"uller
              spaces},
   JOURNAL = {SIGMA Symmetry Integrability Geom. Methods Appl.},
  FJOURNAL = {SIGMA. Symmetry, Integrability and Geometry. Methods and
              Applications},
    VOLUME = {17},
      YEAR = {2021},
     PAGES = {Paper No. 080, 34},
      ISSN = {1815-0659},
   MRCLASS = {30F60 (13F60 58A50)},
  MRNUMBER = {4306793},
MRREVIEWER = {Athanase\ Papadopoulos},
       DOI = {10.3842/SIGMA.2021.080},
       URL = {https://doi.org/10.3842/SIGMA.2021.080},
}

@article {CS14,
    AUTHOR = {Chekhov, Leonid and Shapiro, Michael},
     TITLE = {Teichm\"uller spaces of {R}iemann surfaces with orbifold
              points of arbitrary order and cluster variables},
   JOURNAL = {Int. Math. Res. Not. IMRN},
  FJOURNAL = {International Mathematics Research Notices. IMRN},
      YEAR = {2014},
    NUMBER = {10},
     PAGES = {2746--2772},
      ISSN = {1073-7928,1687-0247},
   MRCLASS = {32G15 (13F60 30F60)},
  MRNUMBER = {3214284},
       DOI = {10.1093/imrn/rnt016},
       URL = {https://doi.org/10.1093/imrn/rnt016},
}

@article {BS24,
    AUTHOR = {Banaian, Esther and Sen, Archan},
     TITLE = {A generalization of {M}arkov numbers},
   JOURNAL = {Ramanujan J.},
  FJOURNAL = {Ramanujan Journal. An International Journal Devoted to the
              Areas of Mathematics Influenced by Ramanujan},
    VOLUME = {63},
      YEAR = {2024},
    NUMBER = {4},
     PAGES = {1021--1055},
      ISSN = {1382-4090,1572-9303},
   MRCLASS = {11D45 (05A19 05E16 11A55 13F60)},
  MRNUMBER = {4721155},
MRREVIEWER = {Hayder\ R.\ Hashim},
       DOI = {10.1007/s11139-023-00801-6},
       URL = {https://doi.org/10.1007/s11139-023-00801-6},
}

@article {FSM12,
    AUTHOR = {Felikson, Anna and Shapiro, Michael and Tumarkin, Pavel},
     TITLE = {Cluster algebras and triangulated orbifolds},
   JOURNAL = {Adv. Math.},
  FJOURNAL = {Advances in Mathematics},
    VOLUME = {231},
      YEAR = {2012},
    NUMBER = {5},
     PAGES = {2953--3002},
      ISSN = {0001-8708,1090-2082},
   MRCLASS = {13F60},
  MRNUMBER = {2970470},
MRREVIEWER = {Kyungyong\ Lee},
       DOI = {10.1016/j.aim.2012.07.032},
       URL = {https://doi.org/10.1016/j.aim.2012.07.032},
}

@misc{evans_arithmetic_2025,
	title = {Arithmetic and geometry of {Markov} polynomials},
	url = {http://arxiv.org/abs/2501.14882},
	doi = {10.48550/arXiv.2501.14882},
	language = {en},
	urldate = {2025-07-22},
	publisher = {arXiv},
	author = {Evans, S. J. and Veselov, A. P. and Winn, B.},
	year = {2025},
	note = {\url{http://arxiv.org/abs/2501.14882}},
}

@misc{Banaian_Gyoda_2025,
      title={Cluster algebraic interpretation of generalized Markov numbers and their matrixizations}, 
      author={Esther Banaian and Yasuaki Gyoda},
      year={2025},
      eprint={2507.06900},
      archivePrefix={arXiv},
      primaryClass={math.CO},
      url={https://arxiv.org/abs/2507.06900}, 
     note={\url{https://arxiv.org/abs/2507.06900}},

}

@misc{EJMGO25,
      title={On $q$-deformed Markov numbers. Cohn matrices and perfect matchings with weighted edges}, 
      author={Sam Evans and Perrine Jouteur and Sophie Morier-Genoud and Valentin Ovsienko},
      year={2025},
      eprint={2507.19080},
      archivePrefix={arXiv},
      primaryClass={math.CO},
      url={https://arxiv.org/abs/2507.19080}, 
    note={\url{https://arxiv.org/abs/2507.19080}},
}
